\documentclass[11pt]{article}
\usepackage{amssymb, latexsym, mathtools, amsthm}
\begingroup
    \makeatletter
    \@for\theoremstyle:=definition,remark,plain\do{%
        \expandafter\g@addto@macro\csname th@\theoremstyle\endcsname{%
            \addtolength\thm@preskip\parskip
            }%
        }
\endgroup
\usepackage{enumerate}
\usepackage{pgf,tikz}
\usepackage{pdfsync}
\usepackage{txfonts}
\usepackage[T1]{fontenc}
\usepackage{booktabs}
 \usepackage{graphicx}
\definecolor{dnrbl}{rgb}{0,0,0.3}
\definecolor{dnrgr}{rgb}{0,0.3,0}
\definecolor{dnrre}{rgb}{0.5,0,0}
\usepackage[colorlinks=true, citecolor=dnrgr, linkcolor=dnrre, urlcolor=dnrre] {hyperref}
\usepackage{xcolor}
\usepackage{parskip}
\usepackage{tabularx, colortbl}
\usepackage{geometry}
 \usepackage[scriptsize, up]{caption}
\usepackage{pgf,tikz}
\usetikzlibrary{decorations, decorations.pathmorphing}
\usetikzlibrary {shapes}

\theoremstyle{plain}
\newtheorem{thm}{Theorem}[section]

\newtheorem{lem}[thm]{Lemma}
\newtheorem{coro}[thm]{Corollary}

\newtheorem{defi}[thm]{Definition}
\numberwithin{equation}{subsection}
\makeatletter

\let\c@table\c@figure
\makeatother

\newcommand{\Nat}{\mathbb{N}}

\newcommand{\restr}{\upharpoonright}  
\newcommand{\de}{\downarrow} 

\DeclarePairedDelimiter{\tuple}{\langle}{\rangle}

\newcommand{\bigo}[1]{\mathop{\bf O}\/\left({#1}\right)}
\newcommand{\smo}[1]{\mathop{\bf o}\/\left({#1}\right)}

\newcommand{\FSW}{Figueira, Stephan, and Wu\ }

\newcommand{\KS}{Ku{\v{c}}era and Slaman\ }
\newcommand{\CHKW}{Calude, Hertling, Khoussainov and Wang\ }
\newcommand{\DHN}{Downey, Hirschfeldt and Nies\ }
\newcommand{\DHL}{Downey, Hirschfeldt and Laforte\ }
\newcommand{\BD}{Bienvenu and Downey\ }
\newcommand{\BMN}{Bienvenu, Merkle  and Nies\ }
\newcommand{\HKM}{H\"{o}lzl, Kr\"{a}ling and Merkle\ }
\newcommand{\CDS}{Calude, Dinneen and Shu\ }

\newcommand{\ml}{Martin-L\"{o}f }

\newcommand{\eg}{e.g.\ }
\newcommand{\ie}{i.e.\ }
\newcommand{\ce}{c.e.\ }

\newcommand{\rce}{right-c.e.\ }

\newcommand{\pf}{prefix-free }

\renewenvironment{abstract}
 { \normalsize
  \list{}{
    \setlength{\leftmargin}{.0cm}%
    \setlength{\rightmargin}{\leftmargin}%
    }%
  \item {\bf \abstractname.} \relax}
 {\endlist}
 
 \makeatletter
\newtheorem*{rep@theorem}{\rep@title}
\newcommand{\newreptheorem}[2]{%
\newenvironment{rep#1}[1]{%
 \def\rep@title{#2 \ref{##1}}%
 \begin{rep@theorem}}%
 {\end{rep@theorem}}}
\makeatother
 \newreptheorem{thm}{Theorem}
\newreptheorem{lem}{Lemma}

\newcommand{\pfm}{prefix-free machine }
\newcommand{\pfmn}{prefix-free machine}

\title{Optimal asymptotic bounds on the oracle use in computations from Chaitin's Omega\thanks{Barmpalias 
was supported by the 
1000 Talents Program for Young Scholars from the Chinese Government,
and the Chinese Academy of Sciences (CAS) President's International 
Fellowship Initiative No. 2010Y2GB03.
Additional support was received by
the CAS and the Institute of Software of the CAS.
Partial support was also received from a Marsden grant of New Zealand 
and the China Basic Research Program (973) grant No.~2014CB340302.
}}
\author{George Barmpalias  \and Nan Fang  \and Andrew Lewis-Pye}
\date{\today}
\begin{document}
\maketitle
\begin{abstract}
Chaitin's number $\Omega$ is the halting probability of a universal \pf
machine, and although it depends on the underlying enumeration
of prefix-free machines, it is always Turing-complete.
It can be observed, in fact, that for every computably enumerable (c.e.) real $\alpha$, there exists a Turing functional via which $\Omega$ computes $\alpha$, and such that the number of bits of $\Omega$ 
that are needed for the computation of the first $n$ bits of $\alpha$ (i.e.\ the \emph{use} on argument $n$)
is bounded above by a computable function $h(n)=n+\smo{n}$.

We characterise the 
asymptotic upper bounds on the use of Chaitin's $\Omega$ in oracle computations
of halting probabilities (\ie c.e.\  reals).
We show that the following two conditions are equivalent for any computable function $h$ such that $h(n)-n$ is non-decreasing: (1)   $h(n)-n$ is an information content measure,
\ie the series $\sum_n 2^{n-h(n)}$ converges, (2) for every c.e.\ real $\alpha$  there exists a Turing functional via which $\Omega$ computes $\alpha$ with use bounded by  $h$. 
We also give a similar characterisation with respect to computations of \ce sets from 
$\Omega$, by showing that the following are equivalent for any computable non-decreasing function $g$: (1) $g$ is an information-content measure, (2)  for every c.e.\ set $A$, $\Omega$ computes $A$ with use bounded by $g$. Further results and some connections with Solovay functions (studied by a number of authors 
\cite{Solovay:75,DBLP:conf/stacs/BienvenuD09, holzlMerkKra09,BienvenuMN11}) 
are given.
\end{abstract}
\vfill
\noindent{\bf George Barmpalias}\\[0.2em]
\noindent State Key Lab of Computer Science, 
Institute of Software, Chinese Academy of Sciences, Beijing, China.
School of Mathematics, Statistics and Operations Research,
Victoria University of Wellington, New Zealand.\\[0.1em] 
\textit{E-mail:} \texttt{\textcolor{dnrgr}{barmpalias@gmail.com}.}
\textit{Web:} \texttt{\href{http://barmpalias.net}{http://barmpalias.net}}
\vfill
\noindent{\bf Andrew Lewis-Pye}\\[0.2em]
\noindent Department of Mathematics,
Columbia House, London School of Economics, 
Houghton Street, London, WC2A 2AE, United Kingdom.\\[0.1em]
\textit{E-mail:} \texttt{\textcolor{dnrgr}{A.Lewis7@lse.ac.uk}.}
\textit{Web:} \texttt{\textcolor{dnrre}{http://aemlewis.co.uk}} 
\vfill
\noindent{\bf Nan Fang}\\[0.2em]
\noindent Institut für Informatik, Ruprecht-Karls-Universität Heidelberg, Germany.\\[0.1em]
\textit{E-mail:} \texttt{\textcolor{dnrgr}{nan.fang@informatik.uni-heidelberg.de}.}
\textit{Web:} \texttt{\textcolor{dnrre}{http://fangnan.org}} 
\vfill\thispagestyle{empty}
\clearpage

\section{Introduction}
Chaitin's constant is the most well-known algorithmically random number.
It is the probability that a universal \pfm $U$ halts, when a random sequence of 0s and 1s is
fed into the input tape. In symbols:
\[
\Omega_U=\sum_{U(\sigma)\de} 2^{-|\sigma|}.
\]
While $\Omega_U$ clearly depends on the choice of underlying universal \pfm $U$, all versions of it 
are quite similar with respect to the properties we shall be concerned with here. For this reason, 
when the choice of underlying machine $U$ is irrelevant, we shall simply refer to $\Omega_U$ as an $\Omega$-number, or 
denote it by $\Omega$.
The choice of $U$ is essentially equivalent to the choice of a programming language. 
Although $\Omega$ is incomputable, it is the limit of an increasing 
computable sequence of rational numbers. Such reals are known as computably
enumerable reals (\ce reals, 
a more general class than the c.e.\ \emph{sets}), and can be viewed as the halting probabilities of
(not necessarily universal) \pf machines.

The binary expansion of $\Omega$ is an algorithmically  unpredictable sequence of 0s and 1s,
which packs a lot of useful information into rather short initial segments.
In the words of Charles H.~Bennett (also quoted by Martin Gardner in \cite{bennet79}), 
\begin{quote}
[Chaitin's constant] embodies an enormous amount of wisdom in a very small space \dots inasmuch as its first few thousands digits, which could be written on a small piece of paper, contain the answers to more mathematical questions than could be written down in the entire universe\dots
\begin{flushright}
Charles H.~Bennett \cite{bennettreal79} 
\end{flushright}
\end{quote}

Being of the same Turing degree, $\Omega$  may be seen as containing the same information as the halting problem, 
but it contains this same information in a much more compact way. It is Turing complete, so it computes all
c.e.\  sets and all c.e.\  reals.
As one might expect, given these remarkable properties, it is very hard to
obtain initial segments of $\Omega$. The modulus of convergence
in any monotone computable approximation to $\Omega$, for example, dominates all computable functions.
Chaitin's incompleteness theorem  \cite{MR0411829} (also independently proved by Muchnik) tells us that
any computably axiomatizable theory interpreting Peano arithmetic
can determine the value of only finitely many bits of $\Omega$.
Solovay \cite{solovay00} exhibited a specific universal prefix-free machine $U$
with respect to which ZFC cannot determine \emph{any} bit of $\Omega_U$.
Despite these extreme negative results, \CDS
\cite{CaludeDS02,ijbc/CaludeD07} computed a finite number of bits for certain versions
of $\Omega$ (\ie corresponding to certain canonical enumerations of the \pf Turing machines).

Given the Turing completeness of $\Omega$ and the extreme difficulty of obtaining the bits of
its binary expansion, it is interesting to consider the lengths of the initial segments of $\Omega$ that
are used in oracle computations of \ce sets and \ce reals.
How many bits of $\Omega$, for example, 
are needed in order to compute $n$ bits of a given c.e.\  real or a given \ce set?
The goal of this paper is to answer these questions, as well raising and discussing
a number of related issues that point to an interesting research direction, with clear connections
to contemporary research in algorithmic randomness.

\subsection{Our main results, in context}\label{erPRfG7zRa}
We review some standard notation and terminology. 
The positions in the binary expansion of a real (to the right of the decimal point)  are
numbered $1,2,3,\dots$ from left to right. 
Given a real $\alpha$, we let $\alpha(t)$ 
denote the value of the digit of $\alpha$
at position $t+1$ of its binary expansion.
We identify infinite binary sequences with binary expansions of reals in $[0,1]$ and with subsets
of the natural numbers $\Nat$. Subsets of $\Nat$ are denoted by upper-case latin letters while 
lower-case Greek letters are used for the reals in $[0,1]$ 
(the fact that dyadic rationals have two expansions will not be an issue here).
The  \emph{use}-function corresponding to an oracle computation of a set $A$, is the function whose value on argument $n$ is the length of the initial segment of the oracle tape that is read during the computation of the first $n$ bits of $A$. 
Given $A\subseteq\Nat$, we let $A\restr_n$ denote the initial segment of the characteristic function of $A$ of length $n$. The $n$-bit prefix of a real number $\alpha$ is 
also denoted by $\alpha\restr_n$. Some background material on algorithmic randomness is given in
Section \ref{ejNBNcNbF}.

\begin{defi}[Use of computations]
Given a function $h$, we say that a set $A$ is computable from a set $B$ with use bounded by $h$
if there exists an oracle Turing machine which, given any $n$ and $B\restr_{h(n)}$ in the oracle tape,
outputs the first $n$ bits of $A$.
\end{defi}

Use functions are typically assumed to be nondecreasing.
We say that a function $f:\Nat\to\Nat$ is \rce if there is a computable function 
$f_0:\Nat\times\Nat\to\Nat$ which is non-increasing in the second argument and such that
$f(n)=\lim_s f_0(n,s)$ for each $n$.
Our results connect use-functions with {\em information content measures}, as these were
defined by Chaitin \cite{chaitinincomp}. These are \rce functions $f$ with the property that
$\sum_i 2^{-f(i)}$ converges. This notion is fundamental in the theory of 
algorithmic randomness, and is the same as
the \ce discrete semi-measures of Levin \cite{Lev74tra}, which were also used in Solomonoff
\cite{Solomonoff:64}. For example, \pf complexity can be characterized as a minimal
information content measure (by \cite{chaitinincomp,Lev74tra}).
Throughout this paper the reader may consider $\log n$ as a typical example
 of a computable non-decreasing function $f$
 such that $\sum_i 2^{-f(i)}$ diverges, 
and $2\log n$ or $(\epsilon +1)\log n $ for any $\epsilon>0$ as
 an example of a function $f$ such that $\sum_i 2^{-f(i)}$ converges.

Our results address computations of (a) \ce sets and (b) \ce reals. In a certain sense 
\ce reals can be seen as compressed versions
of \ce sets. Indeed, as we explain at the beginning of
 Section \ref{XF1vJKLG4s}, a \ce real $\alpha$ can be coded into
a \ce set $X_{\alpha}$ in such a way that $X_{\alpha}\restr_{2^n}$ can compute $\alpha\restr_n$
and $\alpha\restr_n$ can compute $X_{\alpha}\restr_{2^n}$.
So it is not surprising that our results for the two cases are analogous.
Computations of \ce sets (from $\Omega$) 
correspond to use-functions $g$ which are information content measures,
while computations of \ce reals correspond to use-functions of the form $n+g(n)$, where
$g$ is an information content measure.

\begin{thm}[Computing c.e.\  reals]\label{yJAbgkLtG2}
Let $h$ be computable and such that $h(n)-n$ is non-decreasing.
\begin{enumerate}[(1)]
\item If $\sum_n 2^{n-h(n)}$ converges, then for any $\Omega$-number $\Omega$ and any c.e.\ real $\alpha$, $\alpha$ can be computed from $\Omega$ with use bounded by $h$.
This also holds without the monotonicity assumption on $h(n)-n$.  
\item If $\sum_n 2^{n-h(n)}$ diverges  
then no $\Omega$-number can compute all c.e.\  reals
with use $h+\bigo{1}$. In fact, in this case there exist two c.e.\  reals such 
that no c.e.\  real can compute both of them with use $h+\bigo{1}$.
\end{enumerate}
\end{thm}

A considerable amount of work has been devoted to the study of computations 
with use-function $n\mapsto n+c$ for some constant $c$. Note that these are a specific
case of the second clause of Theorem \ref{yJAbgkLtG2}.
These types of reductions are called computably-Lipschitz, and were introduced
by \DHL in \cite{MR2030512} and have been studied in 
\cite{DingY04,DBLP:conf/cie/Barmpalias05, jflBaiasL06,BL05.5,BLibT,MR2286414}
as well as more recently in \cite{IOPORT.05678491,apal/Day10,Merklebarmp, Losert15}.
The second clause of Theorem \ref{yJAbgkLtG2} generalises the result of
Yu and Ding in \cite{DingY04} which assumes that $h(n)=n+\bigo{1}$.
It would be an interesting project to generalise some of the work that has been done for
computably-Lipschitz reductions to Turing-reductions with computable use $h$ such that 
$\sum_i 2^{n-h(n)}$ diverges. We believe that in many cases this is possible.
The second clause of Theorem \ref{yJAbgkLtG2} can be viewed as an example of this larger 
project.

If, as we discussed earlier, \ce reals are just compressed versions of \ce sets, then 
what are the \ce sets that correspond to $\Omega$ numbers?
In Section \ref{ZOJynPYcM2} (with the proofs deferred to Section \ref{XF1vJKLG4s}) 
we discuss this duality and point out that the linearly-complete  \ce sets
could be seen as the analogues of $\Omega$ in the \ce sets.
Recall that the linearly complete c.e.\ sets are those \ce sets that 
compute all \ce sets with use $\bigo{n}$.\footnote{We note that this is different than what was defined
to be linear reducibility in \cite{MR2286414}, which was later more commonly known as computably Lipschitz reducibility.}
A natural example of a linearly-complete set is the set of non-random strings, 
$\{\sigma\ |\ C(\sigma)<|\sigma|\}$ (where $C(\sigma)$ 
denotes the plain Kolmogorov complexity of
the string $\sigma$) which was introduced by Kolmogorov \cite{MR0184801}.
This latter set 
can be viewed as a set of numbers with respect to the bijection
that corresponds to the usual ordering of strings, first by length and then lexicographically.
In \cite{ipl/BarmpaliasHLM13} it was shown that a \ce set $A$ is linearly-complete
if and only if $C(A\restr_n)\geq \log n-c$ for some constant $c$ and all $n$.

\begin{thm}[Computing \ce sets]\label{FmKgYATdsh}
Let $g$ be a computable function.
\begin{enumerate}[(1)]
\item If $\sum_n 2^{-g(n)}$ converges, then every \ce set is computable from any
$\Omega$-number with use $g$. 
\item If $\sum_n 2^{-g(n)}$ diverges and $g$ is nondecreasing, 
then no $\Omega$-number can compute all \ce sets
with use $g$. In fact, in this case, 
no linearly complete \ce set can be computed by any \ce real
with use $g$. 
\end{enumerate}
\end{thm}

Theorem \ref{FmKgYATdsh} is an analogue of
Theorem \ref{yJAbgkLtG2} for computations of \ce sets from $\Omega$.
Beyond the apparent analogy, the two theorems are significantly connected.
Theorem \ref{FmKgYATdsh} will be used in the proof of Theorem \ref{yJAbgkLtG2}, and this
justifies our approach of studying \ce reals as compressed versions of \ce sets. 

A natural class of a linearly complete c.e.\  sets 
are the halting sets, with respect to a {\em Kolmogorov numbering}.
A numbering $(M_e)$ of plain Turing machines is a Kolmogorov numbering if
there exists a computable function $f$ and a machine $U$ such that
$U(f(e,n))\simeq M_e(n)$ for each $e,n$ (where `$\simeq$' means that either both sides are defined
and are equal, or both sides are undefined) and for each $e$ the function $n\mapsto f(e,n)$ is
$\bigo{n}$. Kolmogorov numberings, in a sense, correspond to optimal simulations of Turing machines 
and are the basis for the theory of plain Kolmogorov complexity,
and have been studied in \cite{Lynch1974143,mst/Schnorr75}.

The following result can be viewed in the context of
Solovay functions, which stem from Solovay's work \cite{Solovay:75} on $\Omega$-numbers
and were extensively studied by \BD  \cite{DBLP:conf/stacs/BienvenuD09}, 
\HKM \cite{holzlMerkKra09}, and \BMN
\cite{BienvenuMN11}. A Solovay function is a computable information content measure $g$
such that $\sum_i 2^{-g(i)}$ is 1-random. As we discussed above, every information content measure
is an upper bound on \pf Kolmogorov complexity (modulo an additive constant).
By \cite{BienvenuMN11}, Solovay functions $g$ are exactly the computable {\em tight upper bounds} of
the \pf complexity function $n\mapsto K(n)$, \ie a computable function $g$ is a Solovay function
if and only if $K(n)\leq g(n)+\bigo{1}$ for all $n$ and $g(n)\leq K(n)+\bigo{1}$ for infinitely many $n$.

\begin{thm}[Solovay functions as tight upper bounds on oracle use]\label{1OZQrzfDmp}
Let $g$ be a nondecreasing computable function such that 
$\sum_n 2^{-g(n)}$ converges to a 1-random real. Then
a \ce real is 1-random if and only if it computes the halting problem with 
respect to a Kolmogorov numbering with use $g$.
\end{thm}

This result is another way to see the halting sets with respect to Kolmogorov numberings
as analogues of Chaitin's $\Omega$ in the \ce sets. Moreover, it gives a characterization of
monotone Solovay functions in terms of the oracle-use for certain computations between
\ce sets and \ce reals.

Finally we obtain a gap theorem characterizing the array computable degrees,
in the spirit of Kummer \cite{DBLP:journals/siamcomp/Kummer96}. 
Recall from \cite{DJS2} that a Turing degree $\mathbf{a}$ is array computable if there function $f$
which is computable from the halting problem with a computable upper bound on the use of the oracle,
and which dominates every function $g$ which is computable in $\mathbf{a}$.  Array computable degrees  play a significant role in classical computability theory.
Kummer showed that they also relate to initial segment complexity.
He showed that the initial segment plain complexity $C(A\restr_n)$ of
every \ce set $A$ in an array computable degree  is bounded above (modulo
an additive constant) by any monotone 
function $f$ such that
$\lim_n(f(n)-\log n)=\infty$. On the other hand, he also showed that every array noncomputable degree
contains a \ce set $A$ such that $C(A\restr_n)$ is not bounded above by any function $f$
such that $\lim_n (2\log n-f(n))=\infty$.
This is known as Kummer's gap theorem (see \cite[Section 16.1]{rodenisbook}) since it 
characterizes array computability of the \ce degrees in terms of a logarithmic gap on the 
initial segment plain complexity.
The following is another logarithmic 
gap theorem characterizing array computability in the \ce degrees, but this time
the gap refers to the oracle use in computations from Chaitin's omega.

\begin{thm}\label{gGr7djFsLI}
Let $\mathbf{a}$ be a \ce Turing degree and let $\Omega$ be (an instance of)
Chaitin's halting probability.
\begin{enumerate}[(1)]
\item If $\mathbf{a}$ is array computable, then every real in $\mathbf{a}$
is computable from $\Omega$ with identity use.
\item Otherwise there exists a \ce real in $\mathbf{a}$ which is not computable from $\Omega$
with use $n+\log n$.
\end{enumerate}
\end{thm}
This result is largely a corollary of the work in
\cite[Sections 4.2, 4.3]{IOPORT.05678491} applied to
the construction in the proof of
Theorem \ref{yJAbgkLtG2} (1).

\subsection{Omega numbers and completeness}\label{ZOJynPYcM2}
This paper concerns certain aspects of completeness of $\Omega$, so
it seems appropriate to discuss the completeness of $\Omega$ a little more generally.
The fact that $\Omega$ computes all \ce reals is referred to as the  \emph{Turing-completeness} of $\Omega$
(with respect to c.e.\  reals),
and has been known since $\Omega$ was first defined. Calude and Nies observed in \cite{CaludeN97}
that there is a computable bound on the oracle use in computations of c.e.\  reals and \ce sets 
from $\Omega$.
The results of Section 
\ref{erPRfG7zRa} give a sharp characterisation of these computable bounds.
It is a reasonable question as to whether $\Omega$-numbers can be characterised as the complete
c.e.\  reals with respect to Turing reductions with appropriately bounded use.
In terms of the strong reducibilities  of classical computability theory, this question has a negative answer.
\FSW showed in \cite{jc/FigueiraSW06}, for example,  that there are two $\Omega$-numbers 
which do not have
the same truth-table degree. Stephan (see \cite[Section 6]{IOPORT.05678491} for a proof) showed
that given any $\Omega$-number, there is another $\Omega$-number which is not computable from the
first  with use $n+\bigo{1}$.

Putting this question aside, there are a number of characterisations of the $\Omega$-numbers
as the complete elements of the set of c.e.\  reals with respect to  certain (weak) reducibilities
relating to algorithmic randomness.
\DHL studied and compared these and other reducibilities in \cite{MR2030512}, and a good
presentation of this study can also be found in the monograph of Downey and
Hirschfeldt \cite[Chapter 9]{rodenisbook}.
One of these reducibilities that pertains to the following discussion is the $rK$ reducibility,
denoted by $A \leq_{rK} B$ which is the relation that $K(A\restr_n\ |\ B\restr_n)=\bigo{1}$ for all $n$.
We note that by \DHL \cite{MR2030512}, this relation implies that $A$ is computable in $B$, in the sense of
Turing oracle computability.
The following fact highlights the correspondence between $\Omega$-numbers and linearly complete 
\ce sets. Here $K(\Omega\restr_n\ |\ H\restr_{g(n)})$ denotes the prefix-free complexity of
$\Omega\restr_n$ relative to $H\restr_{g(n)}$, and
$K(\Omega\restr_n\ |\ H\restr_{g(n)})=\bigo{1}$ can be seen as a nonuniform computation
of $\Omega$ from $H$ (see Section \ref{ejNBNcNbF} for definitions relating to complexity).

\begin{thm}\label{6MsEpxPF1}
Consider a 1-random c.e.\  real $\Omega$, a \ce set $H$, a computable function $g$ and a 
\rce function $f$ such that $\sum_i 2^{-f(i)}$ is finite. 
\begin{enumerate}[\hspace{0.5cm}(i)]
\item If $K(\Omega\restr_n\ |\ H\restr_{g(n)})=\bigo{1}$ then $2^n=\bigo{g}$.
\item If $g=\bigo{2^n}$ and  $K(\Omega\restr_{n+f(n)}\ |\ H\restr_{g(n)})=\bigo{1}$ then $H$ is
linearly-complete.
\item If $g=\bigo{n}$ and  $K(\Omega\restr_{f(n)}\ |\ H\restr_{g(n)})=\bigo{1}$ then $H$ is
linearly-complete.
\item If $H$ is linearly-complete then 
it computes $\Omega$ (and any \ce real) with use $2^{n+\bigo{1}}$.
\end{enumerate}
\end{thm}

In order to understand this fact better, note that if $\Omega$ is computed by
a \ce set $H$ with computable 
use $h$, then we have $K(\Omega\restr_n\ |\ H\restr_{h(n)})=\bigo{1}$ for all
$n$. So the first clause of 
Theorem \ref{6MsEpxPF1} says that the 
tightest use in a computation of $\Omega$ from a \ce set is
exponential, even when the computation is done non-uniformly.
The fourth clause asserts that the linearly-complete \ce sets achieve this lower bound
on the use in computations of $\Omega$. In this sense, they can be viewed as
analogues of $\Omega$. There are more facts supporting this suggestion.  
In \cite{ipl/BarmpaliasHLM13}, for example,  it was shown that a \ce set $H$ is linearly complete if and only if
$K(W\restr_n\ |\ H\restr_n)=\bigo{1}$ for all $n$ and all \ce sets $W$. 
By the basic properties of Kolmogorov complexity, \pf complexity in
the previous statement and in Theorem \ref{6MsEpxPF1} can be replaced by plain
Kolmogorov complexity (\eg see \cite{MR2030512}).

It is also instructive to compare clauses (ii) and (iii) of Theorem \ref{6MsEpxPF1} with
an older result of Solovay.  For each $n$ let $D_n$ be the set of strings of length at most $n$,
in the domain of the universal \pfmn. Solovay \cite{Solovay:75}
(see  \cite[Section 3.13]{rodenisbook}) showed that the cardinality of $D_n$ is proportional to 
$2^{n-K(n)}$, that 
$K(\Omega\restr_n\ |\ D_{n+K(n)})=\bigo{1}$ 
and that $K(D_n\ |\ \Omega\restr_n)=\bigo{1}$.

We conclude this section with a brief discussion of some results from Tadaki \cite{Tadaki:2009_72}
which came to our attention only recently. Tadaki's results are incomparable with those presented here, but relate to the results from this paper concerning reductions between $\Omega$ and c.e.\ sets (rather than c.e.\  reals).   He
considered the problem of how many bits of $\Omega$ are needed in order to compute
the domain of the universal \pf machine $U$ up to the strings of length $n$; and vice-versa, what is the
least number $m$ such that, in general, we can compute the first $n$ bits of $\Omega$ from the
domain of $U$ restricted to the strings of length at most $m$.
Tadaki showed that, if we restrict the question to computable use-functions, 
\begin{enumerate}[\hspace{0.5cm}(a)]
\item the answer to the first question is given by the computable functions of the type $n-f(n)+\bigo{1}$ 
such that $\sum_n 2^{-f(n)}$ is finite.
\item the answer to the second question is given by the computable functions of the type $n+f(n)+\bigo{1}$ 
such that $\sum_n 2^{-f(n)}$ is finite.
\end{enumerate}
These optimal results are quite pleasing, and were subsequently used by Tadaki \cite{tadaki2011rep}
in order to obtain a statistical mechanics interpretation of algorithmic information theory, which
was further developed in \cite{DBLP:conf/itw/Tadaki11, DBLP:journals/mscs/Tadaki12}.
 The main difference here is that
the \ce sets considered by Tadaki are the domains of universal \pf machines, and these are
not linearly-complete (a fact which is not hard to verify).

\subsection{Preliminaries}\label{ejNBNcNbF}
We assume the reader is familiar with the basic concepts of Kolmogorov complexity and computability theory.
We use $C(\sigma)$ and $K(\sigma)$ to denote the plain and 
\pf Kolmogorov complexity of a string $\sigma$ respectively. Moreover we use
$C(\sigma\ |\ \rho), K(\sigma\ |\ \rho)$ to denote  the relative plain and \pf complexities of $\sigma$ 
with respect to $\rho$ respectively.
Recall that $A$ is 1-random if  there exists a constant $c$
such that $K(A\restr_n)\geq n-c$ for all $n$.
The cumulative work of Solovay \cite{Solovay:75},
\CHKW \cite{Calude.Hertling.ea:01} and \KS \cite{Kucera.Slaman:01}
has  shown that 
the 1-random c.e.\  reals are exactly the halting probabilities of
(optimal) universal machines.
Most of this work revolves around a reducibility between \ce reals that was introduced in 
\cite{Solovay:75}. We say that a \ce real $\alpha$ is Solovay reducible to \ce real $\beta$ if
there are nondecreasing computable 
rational approximations $(\alpha_s), (\beta_s)$ to $\alpha,\beta$ respectively, and a constant $c$ such
that $\alpha-\alpha_s\leq c\cdot (\beta-\beta_s)$ for all $s$. \DHN \cite{DHN} showed that 
\begin{equation}\label{pI7DoArmo}
\parbox{14.5cm}{$\alpha$
is Solovay reducible to $\beta$ iff   
there are nondecreasing computable 
rational approximations $(\alpha_s), (\beta_s)$ to $\alpha,\beta$ respectively, and a constant $c$ such
that $\alpha_{s+1}-\alpha_s\leq c\cdot (\beta_{s+1}-\beta_s)$ for all $s$.}
\end{equation}
Moreover it follows from
\cite{Solovay:75,Calude.Hertling.ea:01,Kucera.Slaman:01} that the 1-random \ce 
reals are exactly the \ce reals that are complete (\ie they are above all other \ce reals) with respect to
Solovay reducibility. In particular, if $\alpha$ is a 1-random \ce real which is Solovay reducible to 
another \ce real $\beta$, then $\beta$ is also 1-random. We use these facts in the proof
of Theorem \ref{FmKgYATdsh}
in Section \ref{iX9NpUznYr}.
Without loss of generality, all \ce reals considered in this paper belong to $[0,1]$.
The following convergence test will be used  throughout this paper.
\begin{lem}[Condensation]\label{eq:condtest}
If $f$ is a nonincreasing sequence of positive reals then the series
$\sum_i f(i)$ converges if and only if
the series $\sum_i \big(2^i\cdot f(2^i)\big)$ converges.
Moreover, if $f$ is computable and the two sums converge, the first sum is 1-random
if and only if the second sum is 1-random.
\end{lem}
\begin{proof}
The first part is a standard series convergence test known as the Cauchy condensation test.
Its proof usually goes along the lines of Oresme's
proof of the divergence of the harmonic series which shows that
\[
\sum_{i=2^t}^{\infty} f(i)\leq
\sum_{i=t}^{\infty} \big(2^i\cdot f(2^i)\big)=
2 \cdot \sum_{i=t}^{\infty}\big(2^{i-1} \cdot f(2^i)\big)
\leq 2\cdot \sum_{i=2^{t-1}}^{\infty} f(i)
\hspace{0.5cm}
\textrm{for all $t\in \Nat^+$.}
\]
 The second part of the lemma follows from the above bounds,
which show that the two sums are \ce reals, each one Solovay reducible to the other.
So one is 1-random if and only if the other is.
 \end{proof}

Algorithmic randomness for reals was originally defined in terms of effective
statistical tests by \ml in \cite{MR0223179}. We assume that the reader is familiar with this classical
definition.
Another kind of test that can be used for the definition of 1-randomness is the Solovay test.
A Solovay test is a uniform sequence $(I_i)$ of $\Sigma^0_1$ classes
(often viewed as a uniformly \ce sequence of sets of binary strings) such that the sum of the measures
of the sets $I_i$ is finite. We say that a real passes a Solovay test $(I_i)$ if there are only finitely many
indices $i$ such that $I_i$ contains a prefix of $X$. Solovay  \cite{Solovay:75} showed that
a real is 1-random if and only if it passes all Solovay tests.

Some of the following proofs involve the construction of prefix-free machines.  
There are certain notions
and tools associated with such constructions which are standard in the literature, and which we briefly review now.
The {\em weight} of a prefix-free set $S$ of strings 
is defined
to be the sum $\sum_{\sigma\in S}  2^{-|\sigma|}$. The
{\em weight} of a prefix-free machine $M$ is 
defined to be the weight of its domain. 
Prefix-free machines are most often built in terms of {\em request sets}.
A request set $L$ is a set of pairs $\langle \rho, \ell\rangle$ where $\rho$ is a string
and $\ell$ is a positive integer. A `request' $\langle \rho, \ell\rangle$
represents the intention of describing $\rho$
with a string of length $\ell$. We
define the {\em weight of the request} $\langle \rho, \ell\rangle$ to be $2^{-\ell}$.
We say that $L$ is a {\em bounded request set}
if  the sum of the weights of the requests in $L$ is less than 1.
The Kraft-Chaitin theorem
(see e.g.\ \cite[Section 2.6]{rodenisbook}) says that for every bounded request set $L$
which is c.e., there exists a prefix-free machine $M$ with the property that for 
each $\langle \rho, \ell\rangle\in L$
there exists a string $\tau$ of length $\ell$ such that $M(\tau)=\rho$.
So the dynamic construction of a prefix-free machine can be reduced to a mere
description of a corresponding c.e.\ bounded request set.

For more background in algorithmic information theory and computability theory
we refer to the standard monographs \cite{MR1438307,rodenisbook, Ottobook}.
Since we have included a number of citations to the
unpublished manuscript of Solovay \cite{Solovay:75}, we note that every result in this manuscript
has been included, with a proof, in the monograph by Downey and Hirschfeldt \cite{rodenisbook}.

\section{Omega and the computably enumerable sets}\label{XF1vJKLG4s}
We start with the proof of Theorem \ref{6MsEpxPF1}
and continue with the proof of Theorem \ref{FmKgYATdsh}.
The proof of the last clause of Theorem \ref{6MsEpxPF1}
in Section \ref{RwuumeDoKH} describes a compact coding of a \ce real into a \ce set, which
will also be used in later arguments in this paper.

\subsection{Proof of Theorem \ref{6MsEpxPF1}, clauses (i) and (ii)}
Recall the following theorem.
\begin{repthm}{6MsEpxPF1}[Clauses (i) and (ii)]
Consider a 1-random c.e.\  real $\Omega$, a \ce set $H$, a computable function $g$ and a 
\rce function $f$ such that $\sum_i 2^{-f(i)}$ is finite. 
\begin{enumerate}[\hspace{0.5cm}(i)]
\item If $K(\Omega\restr_n\ |\ H\restr_{g(n)})=\bigo{1}$ then $2^n=\bigo{g}$.
\item If $g=\bigo{2^n}$ and  $K(\Omega\restr_{n+f(n)}\ |\ H\restr_{g(n)})=\bigo{1}$ then $H$ is
linearly-complete.
\end{enumerate}
\end{repthm}

For the first clause, it suffices to show that for each c.e.\ real $\alpha$, each \ce set $H$ and each
computable function $g$: 
\[
\parbox{12cm}{if $K(\alpha\restr_n\ |\ H\restr_{g(n)})=\bigo{1}$ and $2^n\neq\bigo{g}$ then
$\alpha$ is not 1-random.}
\]
Assuming the above properties regarding $H,\alpha$ and $g$, we construct a \ml test $(V_i)$ 
such that $\alpha\in\cap_i V_i$.
Let $d$ be a constant such that
$K(\alpha\restr_n\ |\ H\restr_{g(n)})<d$ for all $n$.
Since $2^n\neq\bigo{g}$, for each constant $c$ there exists some $n$ such that
$g(n)<2^{n-c}$. Let $(n_i)$ be an increasing sequence with 
$g(n_i)<2^{n_i-i-d}$ for each $i$. 
We describe a construction enumerating the sets $V_k$.  We say that $V_k$ requires attention 
at stage $s+1$ if  $K_{s+1}(\alpha_{s+1}\restr_{n_k}\ |\ H_{s+1}\restr_{g(n_k)})<d$
and either $V_k[s]$ is empty or
the last string enumerated into $V_k$ is not a prefix of $\alpha_{s+1}$.
At each stage $s+1$ we pick the least $k$ such that $V_k$ requires attention and enumerate
$\alpha_{s+1}\restr_{n_k}$ into $V_k$. In order to see that $(V_i)$ is a \ml test, note that each time we enumerate into $V_k$,  we are guaranteed
a change in $\alpha\restr_{n_k}$. 
Moreover, after at most $2^d$ such consecutive enumerations, there must be a change in $H\restr_{g(n_k)}$.
This is true, because after each such enumeration,  $V_k$ can only require attention if the underlying
universal machine issues a new description of $\alpha\restr_{n_k}$ relative to $H\restr_{g(n_k)}$, of
length less than $d$; moreover the underlying machine 
can only issue at most $2^d$ such descriptions as long as the approximation to $H\restr_{g(n_k)}$
does not change.
Since $g(n_k)<2^{n_k-k-d}$ and $H$ is a \ce set, 
there can be at most $2^{n_k-k-d}$ changes in $H\restr_{g(n_k)}$.
It follows that there can be at most  $2^d\cdot 2^{n_k-k-d}=2^{n_k-k}$ many enumerations into $V_k$, the last enumeration into this set being 
 a prefix of $\alpha$.
Since each string in $V_k$ has length $n_k$, the measure of $V_k$ is 
at most $2^{-k}$.
So $(V_k)$ is a \ml test with $\alpha\in\cap_i V_i$, 
demonstrating that $\alpha$ is not 1-random.

For the second clause, 
let $W$ be a \ce set and assume the hypothesis of the second clause of the theorem.  
Since $g=\bigo{2^n}$ it suffices to show that for all but finitely many $k$, and for all $i<2^k$
we can compute $W(2^k+i)$ uniformly from $H\restr_{g(k)}$.
Let $c$ be a constant such that $K(\Omega\restr_{f(k)+k}\ |\ H\restr_{g(k)})< c$ for all $k$.
We enumerate a Solovay test $(I_k)$ as follows. 
We define the sets $I_k(i)$ for each $k\in\Nat$ and $i<2^k$, 
and for each $k$ we let $I_k$ be the union of all
$I_k(i)$, $i<2^k$.
In what follows, when we write $\Omega_{s}\restr_{f(k)+k}$, it is to be understood that the value $f(k)$ referred to, is in fact the approximation to $f(k)$ at stage $s$. 
At any stage $s$ of the construction, for each $k$ and each $i<2^k$, 
let $m_k(i)$ be the first stage $\leq s$ at which a string was enumerated into $I_k(i)$ if such a stage exists, and let $m_k(i)$ be undefined otherwise.
We say that $I_k(i)$ requires attention at
stage $s$ if $i<2^k$, $2^k+i\in W_s$,  
$K_{s}(\Omega_{s}\restr_{f(k)+k}\ |\ H_{s}\restr_{g(k)})< c$ and 
the following two conditions hold:
\begin{itemize}
\item $H_{s}\restr_{g(k)}=H_{m_k(i)}\restr_{g(k)}$ or $m_k(i)$ is undefined;
\item $I_k=\emptyset$ or the last string enumerated into $I_k$ is not a prefix of 
$\Omega_{s}$.
\end{itemize}
We say that $I_k$ requires attention at stage $s$ if  $I_k(i)$ requires attention at stage $s$
for some $i<2^k$. 

At stage $s+1$ 
let $k$ be the least number
such that $I_k$ requires attention. Let $i$ be the least number which is less than $2^k$ 
and such that $I_k(i)$ requires attention, and enumerate
$\Omega_{s+1}\restr_{k+f(k)}$ into $I_k(i)$ (or do nothing if no $I_k$ requires attention). 
This concludes the enumeration of $(I_k)$.

First we show that $(I_k)$ is a Solovay test. 
Fix $k$. If $s_0<s_1$ are two stages at which 
enumerations are made into $I_k$, we have that
$\Omega_{s_0}\restr_{k+f(k)}\neq \Omega_{s_1}\restr_{k+f(k)}$ (if the approximation to $f$ changes, then these are distinct strings).
Now let $i<2^k$.
If $s_0<s_1$ are two stages at which
enumerations into $I_k(i)$ are made then
$K_{s_0}(\Omega_{s_0}\restr_{k+f(k)}\ |\ H_{s_0}\restr_{g(k)})< c$ and 
$K_{s_1}(\Omega_{s_1}\restr_{k+f(k)}\ |\ H_{s_0}\restr_{g(k)})< c$.
So for each new enumeration into $I_k(i)$, we have an additional description
in the universal machine with oracle $H_{m_k(i)}\restr_{g(k)}$
(where $m_k(i)$ is the first stage where an enumeration occurred in $I_k(i)$)
of length less than $c$.
Since there are at most $2^c$ such descriptions, for each $k$ and $i<2^k$ there can be at most $2^c$
enumerations into $I_k(i)$.  So the measure of $I_k(i)$ is at most $2^{c-k-f(k)}$.
Hence the measure of $I_k$ is at most  $2^{c-f(k)}$. From our hypothesis concerning
$f$,  it follows that $(I_k)$ is a Solovay test.

Now we show that $W$ is linearly reducible to $H$. Since $g=\bigo{2^n}$,
it suffices to show that for
all but finitely many $k$, if $i<2^k$ and $2^k+i$ 
is enumerated into $W$ at some stage $s_0$, then there exists a stage
$s_1> s_0$ such that $H_{s_0}\restr_{g(k)}\neq H_{s_1}\restr_{g(k)}$.
Since $\Omega$ is 1-random and each of the sets $I_k(i)$ is a finite set of strings, 
there exist some $k_0$ such that for each $k>k_0$ and each $i<2^{k}$ the
set $I_k(i)$ does not contain any prefix of $\Omega$. Now suppose that some number $2^k+i$ 
with 
$k>k_0$ and $i<k$ is enumerated into $W$ at some stage $s_0$, and towards
a contradiction suppose that $H_{s_0}\restr_{g(k)}= H\restr_{g(k)}$.
Then $I_k(i)$ will require attention at every sufficiently late stage $s$ at which $\Omega_s$ does not have a prefix
in $I_k$.  Eventually $\Omega\restr_{k+f(k)}$ will be enumerated into $I_k(i)$, which contradicts 
the choice of $k_0$. This concludes the proof that $H$ is linearly-complete, and the proof of
the second clause.

\subsection{Proof of Theorem \ref{6MsEpxPF1}, clauses (iii) and (iv)}\label{RwuumeDoKH}
Recall the following theorem.
\begin{repthm}{6MsEpxPF1}[Clauses (iii) and (iv)]
Consider a 1-random c.e.\  real $\Omega$, a \ce set $H$, a computable function $g$ and a 
\rce function $f$ such that $\sum_i 2^{-f(i)}$ is finite. 
\begin{enumerate}[\hspace{0.5cm}(i)]
\setcounter{enumi}{2}
\item If $g=\bigo{n}$ and  $K(\Omega\restr_{f(n)}\ |\ H\restr_{g(n)})=\bigo{1}$ then $H$ is
linearly-complete.
\item If $H$ is linearly-complete then 
it computes $\Omega$ (and any \ce real) with use $2^{n+\bigo{1}}$.
\end{enumerate}

\end{repthm}

For  clause (iii) the proof is similar to clause (ii).  Let $c$ be a constant such that $g(n)\leq c\cdot n$ and
$K(\Omega\restr_{f(n)}\ |\ H\restr_{g(n)})< c$ for all $n$.
We enumerate a Solovay test $(I_k)$ as follows. 
At any stage $s$ of the construction, for each $k$ let $m_k$ be the first stage $\leq s$ at which  a string was enumerated into $I_k$ if such a stage exists, and let $m_k$ be undefined otherwise.
We say that $I_k$ requires attention at
stage $s$ if $k\in W_s$,  $K_{s}(\Omega_{s}\restr_{f(k)}\ |\ H_{s}\restr_{g(k)})< c$,
$H_{s}\restr_{g(k)}=H_{m_k}\restr_{g(k)}$ or $m_k$ is undefined, 
and: 
\[
\textrm{$I_k=\emptyset$ or the last string enumerated in $I_k$ is not a prefix of 
$\Omega_{s}$.} 
\]
At stage $s+1$ 
 let $k$ be the least number
such that $I_k$ requires attention and enumerate
$\Omega_{s+1}\restr_{f(k)}$ into $I_k$ (or do nothing if no $I_k$ requires attention). This concludes the enumeration of $(I_k)$.

First we show that $(I_k)$ is a Solovay test. 
If $s_0<s_1$ are two stages at which 
enumerations into  $I_k$ are made, we have
$K_{s_0}(\Omega_{s_0}\restr_{f(k)}\ |\ H_{s_0}\restr_{g(k)})< c$,
$K_{s_1}(\Omega_{s_1}\restr_{f(k)}\ |\ H_{s_0}\restr_{g(k)})< c$
and $\Omega_{s_0}\restr_{f(k)}\neq \Omega_{s_1}\restr_{f(k)}$.
Hence for each new enumeration into $I_k$, we have an additional description
in the universal machine with oracle $H_{m_k}\restr_{g(k)}$
(where $m_k$ is the first stage where an enumeration occurred in $I_k$)
of length less than $c$.
Since there are at most $2^c$ such descriptions, for each $k$ there can be at most $2^c$
enumerations into $I_k$.  So the measure of $I_k$ is at most $2^{c-f(k)}$. By our hypothesis concerning
$f$ it follows that $(I_k)$ is a Solovay test.

Now we show that $W$ is linearly reducible to $H$. It suffices to show that for
all but finitely many $n$, if $n$ is enumerated into $W$ at some stage $s_0$ then there exists a stage
$s_1> s_0$ such that $H_{s_0}\restr_{g(n)}\neq H_{s_1}\restr_{g(n)}$.
Since $\Omega$ is \ml random, there exist some $n_0$ such that for each $n>n_0$ the
set $I_n$ does not contain any prefix of $\Omega$. Suppose that some
$n>n_0$ is enumerated into $W$ at some stage $s_0$,  and towards 
a contradiction suppose that $H_{s_0}\restr_{g(n)}= H\restr_{g(n)}$.
Then $I_n$ will require attention at every sufficiently late stage $s$ at which $\Omega_s$ does not have a prefix
in $I_n$. Eventually $\Omega\restr_{f(n)}$ will be enumerated into $I_n$, which contradicts 
the choice of $n_0$. This concludes the proof that $H$ is linearly-complete, and the proof of
the third clause.

For clause (iv), let us start by examining how $\Omega$ may be coded into a c.e.\ set.
Given a \ce real $\alpha$, there is a computable increasing sequence of rational numbers
$(\alpha_s)$ that converges to $\alpha$. For each $t\in\Nat$ let: 
\begin{eqnarray*}
p_{\alpha}(t)&=&|\{s\ |\ \alpha_s(t)\neq \alpha_{s+1}(t)\ \wedge\ \alpha_{s+1}(t)=1\}|\leq 2^{t}\\
\sigma_{\alpha}(t)&=& 1^{p_{\alpha}(t)}\ast 0^{2^t-p_{\alpha}(t)}.
\end{eqnarray*}
The \ce real $\alpha$ can be coded into a \ce set $X_{\alpha}$ as follows:
\begin{equation}\label{MvpKJLO4s}
X_{\alpha}=\sigma_{\alpha}(0)\ast \sigma_{\alpha}(1)\ast\cdots
\end{equation}

Then clearly $\alpha$ is computable from $X_{\alpha}$ with use bounded by  $2^n$, and $X_{\alpha}$ is
computable from $\alpha$ with use bounded by $\lfloor \log n \rfloor +1$.
From this coding, we can deduce the fourth clause of
the theorem.
Let $\alpha$ be a \ce real and consider $X_{\alpha}$ as above.
Since $H$ is linearly-complete, $X_{\alpha}\restr_{2^n}$ is 
uniformly computable from $H\restr_{2^{n+\bigo{1}}}$. Since $\alpha\restr_n$ 
is uniformly computable from $X_{\alpha}\restr_{2^n}$, it follows that it is also 
uniformly computable from $H\restr_{2^{n+\bigo{1}}}$.

\subsection{Proof of Theorem \ref{FmKgYATdsh}}\label{iX9NpUznYr}
Recall the following theorem.
\begin{repthm}{FmKgYATdsh}
Let $g$ be a computable function.
\begin{enumerate}[(1)]
\item If $\sum_n 2^{-g(n)}$ converges, then every \ce set is computable from any
$\Omega$-number with use $g$. 
\item If $\sum_n 2^{-g(n)}$ diverges and $g$ is nondecreasing, 
then no $\Omega$-number can compute all \ce sets
with use $g$. In fact, in this case, 
no linearly complete \ce set can be computed by any \ce real
with use $g$. 
\end{enumerate}
\end{repthm}

For clause (1), let $A$ be a \ce set and
 let $c$ be a constant such that $\sum_{i>c} 2^{-g(i)}<1$.
We construct a \pfm $N$ using the Kraft-Chaitin theorem as follows.
At stage $s+1$,  
for each $n\in A_{s+1}-A_s$ which is larger than $c$, 
we enumerate an $N$-description of $\Omega_{s+1}\restr_{g(n)}$ 
(chosen by the online Kraft-Chaitin algorithm) of length
$g(n)$. This completes the definition of $N$.
Note that the weight of the domain of $N$ is 
at most $\sum_{n>c} 2^{-g(n)}<1$,  so that the machine $N$ is well defined by the Kraft-Chaitin theorem. 
Since $\Omega$ is 1-random, there exists some $n_0>c$ such that for all $n>n_0$ we have
$K_N(\Omega\restr_{g(n)})>g(n)$. Hence if some $n>n_0$ enters $A$ at some stage $s+1$,
we have that $\Omega_{s+1}\restr_{g(n)}$ is not a prefix of $\Omega$. Since $\Omega$ is a \ce real,
this means that $\Omega$ computes $A$ with use $g$, which
concludes the proof of  clause (1) of Theorem \ref{FmKgYATdsh}.

We first prove clause (2) of Theorem \ref{FmKgYATdsh}
for the case when the linearly-compete set is the halting set with respect to some Kolmogorov numbering.
We do this because the proof of Theorem \ref{1OZQrzfDmp} is based on the same ideas.
On the other hand, the general case can easily be obtained
with some modifications of the argument which we lay out at the end of this section.
So
suppose
that $A$ is the halting
problem with respect to a Kolmogorov numbering $(M_e)$ of Turing machines. 
Let $(\Phi_e, \alpha_e)$ be an effective list of all Turing functionals $\Phi_e$ with use $g$ and 
\ce reals $\alpha_e$.  This effective list comes with nondecreasing computable rational 
approximations $(\alpha_e[s])$ to $\alpha_e$ and effective enumerations
$\Phi_e[s]$ of $\Phi_e$, which are based on the underlying universal machine.
Also let $U$ be the universal machine
whose halting problem is $A$, and let $f$ be a computable function such that for each $e$ the function 
$n\mapsto f(e,n)$ is $\bigo{n}$ and such that
$U(f(e,n))\simeq M_e(n)$ for all $e,n$. 
By definition, we have $A=\{f(e,n)\ |\ M_e(n)\de\}$.
We will define a Turing machine $M$
such that
the following requirements are met:
\[
\mathcal{R}_e:\ A\neq\Phi_e^{\alpha_e}.
\]
Note that $M$ is not directly mentioned in the requirement $\mathcal{R}_e$, but is implicit in the definition of $A$.
By the recursion theorem we can assume given $b$ such that $M=M_b$. During the construction of $M$ we check the enumeration of $A$ to ensure that there does not exist any $n$ with $A(f(b,n))=1$ and for which we have not yet defined $M$ on argument $n$. If such an $n$ is found, then we terminate the construction of $M$ (so that, in fact, no such $n$ can be found at any stage for the correct index $b$). It is also convenient to speed up the enumeration of $A$, so that whenever we define $M$ on argument $n$, $f(b,n)$ is enumerated into $A$ at the next stage.

Since $\sum_i 2^{-g(i)}=\infty$, by Lemma \ref{eq:condtest} we also have
$\sum_n 2^{n-g(2^n)}=\infty$. Let $I_t=[2^t-1,2^{t+1}-1)$ so that $|I_t|=2^t$.
Let $(J_c(e))$ be a computable partition of $\Nat$
into consecutive intervals (i.e.\ such that $J_c(e)$ and $J_{c'}(e')$ are disjoint if $e\neq e'$ or $c\neq c'$), such that:
\begin{equation}\label{rLVla3Ivpd}
\sum_{t\in J_c(e)} 2^{(t+c)-g(2^{(t+c)})}> 2^c
\hspace{0.5cm}\textrm{for each $e,c$}.
\end{equation}
Such a partition exists by the hypothesis for $g$ and
Lemma \ref{eq:condtest}.
Let $M[s]$ denote the machine $M$ as defined by the end of stage $s$, i.e. $M(n)[s]\downarrow$ precisely if an axiom to that effect is enumerated prior to the end of stage $s$. 
For each pair $c,e$ let $j_{c,e}$ be the largest number $j$ such that there exists $t\in J_c(e)$ with $j\in I_t$. 
We say that $R_e$ requires attention with respect to $(c,t)$ at stage $s+1$, if $t$ is the least number in 
$J_c(e)$ such that $M(i)[s]$ is not defined for some $i\in I_t$,
and if also  $\Phi_e^{\alpha_e}(j)[s+1]=A(j)[s+1]$ for all $j\leq e\cdot j_{c,e}$. We say
$\mathcal{R}_e(c)$ requires attention with respect to $t$ at some stage  if $R_e$ requires attention at that stage
with respect to $(c,t)$. 
Let $\tuple{c,e}$ be an effective bijection between $\Nat\times\Nat$ and $\Nat$.
The definition of $M$ is as follows. At stage $s+1$ we look for the least $\tuple{c,e}$ such that
$\mathcal{R}_e(c)$ requires attention. If there exists no such $\tuple{c,e}<s$, we go to the next stage. 
Otherwise let $\tuple{c,e}$ be this tuple and let $t$ be the least number such that
$\mathcal{R}_e(c)$ requires attention with respect to $t$. 
Let $k$ be the least  element of $I_t$ such that $M_s(k)$ is undefined
and define $M_{s+1}(k)=0$. We say that $\mathcal{R}_e(c)$ received attention with respect to $t$
at stage $s+1$. This concludes the definition of $M$.

It remains to verify that each $\mathcal{R}_e$ is met. Towards a contradiction, suppose that $R_e$ is not met.
Let $b$ be an index of the machine $M$ and fix $c$ to be a constant such that
$f(b,n)\leq 2^{c-1}\cdot n$ for all $n$.
By the padding lemma we can assume that $e>2^{c-1}$, which means that if $\mathcal{R}_e(c)$ requires attention at stage $s+1$ then:
\begin{equation}\label{fhJ4lloHiI}
\textrm{$\Phi_e^{\alpha_e}(f(b,i))[s+1]\de=A(f(b,i))[s+1]$ for all $i\in I_t$ and all $t\in J_{c}(e)$.}
\end{equation}
Fix $t\in J_{c}(e)$ and let $s_0< s_1$ be stages at which 
$R_e(c)$ receives  attention with respect to $t$. Since all elements of $I_t$ are less than $2^{t+1}$, $g$ is nondecreasing,  and since we define $M$ on $i\in I_t$ at stage $s_0$, causing $f(b,i)$ to be enumerated into $A$, it follows that we must see an increase in $\alpha_e\restr_{g(2^{c-1}\cdot 2^{t+1})} = \alpha_e \restr_{g(2^{t+c})}$ between stages $s_0$ and $s_1$. 
If $\mathcal{R}_e$ is not satisfied,
as we have assumed, then for each $t\in J_c(e)$ it will receive attention $|I_t|=2^t$ many times, causing a total increase in $\alpha_e\restr_{g(2^{t+c})}$ of $2^{t-g(2^{t+c})}$. 
Summing over all $t\in J_c(e)$ this means that ultimately:
\[
\alpha_{e}\geq \sum_{t\in J_c(e)} 2^{t-g(2^{t+c})}=
2^{-c}\cdot \sum_{t\in J_c(e)} 2^{t+c-g(2^{t+c})}
>2^{-c}\cdot 2^c=1.
\]
This gives the required contradiction, since $\alpha_e\leq 1$. 
We can conclude that each $\mathcal{R}_e$ is met, which completes the proof of
clause (2) of Theorem \ref{FmKgYATdsh} for the case of halting sets.

For the more general case, assume that $A$ is merely a linearly-complete \ce set.
Then for all \ce sets $W$ we have $K(W\restr_n\ |\ A\restr_n)=\bigo{1}$ (see  \cite{ipl/BarmpaliasHLM13}).
As before, let $(\alpha_e)$ be an effective enumeration of all \ce reals in $[0,1]$,
with nondecreasing rational computable approximations $(a_e[s])$ respectively.
We want to show that for each $e$, the real $\alpha_e$ does not compute $A$ with oracle-use
bounded above by $g$. If $\alpha_e$ did compute $A$ in this way, then
we would have $K(A\restr_n\ |\ \alpha_e\restr_{g(n)})=\bigo{1}$ for all $n$, and since $A$ is
linearly complete, it follows from the remark above that for each \ce set $W$ we would have
$K(W\restr_{2^n}\ |\ \alpha_e\restr_{g(2^n)})=\bigo{1}$ for all $n$.
 So
it suffices to enumerate a \ce set $W$ such that
the following conditions are met for each $e,c$:
\[
\mathcal{R}_e(c):\hspace{0.5cm}
\textrm{If for all $n$,\hspace{0.3cm}$K(W\restr_{2^n}\ |\ \alpha_e\restr_{g(2^n)})\leq c$,
\hspace{0.3cm} then \hspace{0.3cm}$\alpha_e> 1$.}
\]
As before
we let $I_t=[2^t-1,2^{t+1}-1)$ for each $t$, so that $|I_t|=2^t$ and all elements of $I_t$ are less than $2^{t+1}$. 
We also define an appropriate version of the intervals $J_c(e)$ as follows. We consider a 
computable partition of
$\Nat$ into finite intervals $J_c(e)$ for $e,c\in\Nat$ such that
\begin{equation}\label{gmvubml2G4}
\sum_{t\in J_c(e)} 2^{t-g(2^{t+1})}>2^c
\hspace{0.5cm}
\textrm{for each $e,c\in\Nat$.}
\end{equation}
As before, such a partition of $\Nat$ exists by  Lemma
\ref{eq:condtest}.
We say that $\mathcal{R}_e(c)$ 
requires attention with respect to $t\in J_c(e)$ 
at stage $s+1$, if  
$I_t- W_s\neq\emptyset$ 
and if  
$K_{s+1}(W_s\restr_{2^{t'+1}}\ |\ \alpha_e[s+1]\restr_{g(2^{t'+1})})\leq c$ for all $t'\in J_c(e)$.
We say that $\mathcal{R}_e(c)$ 
requires attention at stage $s+1$ if it requires attention with respect to some $t\in J_c(e)$. 
The enumeration of $W$ takes place as follows. At stage $s+1$ we look for the least $\tuple{c,e}$ such that
$\mathcal{R}_e(c)$ requires attention. If there exists no such $\tuple{c,e}<s$, we go to the next stage. 
Otherwise let $\tuple{c,e}$ be this tuple and let $t$ be the least number such that
$\mathcal{R}_e(c)$ requires attention with respect to $t$. 
Let $k$ be the least  element of $I_t-W_s$ and enumerate $k$
into $W_{s+1}$. We say that $\mathcal{R}_e(c)$ received attention with respect to $t$
at stage $s+1$. This concludes the enumeration of $W$.

It remains to verify that each $\mathcal{R}_e$ is met. 
Towards a contradiction, suppose that $R_e$ is not met.
Fix $t\in J_{c}(e)$ and for each $i\leq 2^{t-c}$ let $s_i$ be
the stage where
$\mathcal{R}_e(c)$ receives attention with respect to $t$ for the $(2^c\cdot i)$-th time. 
Since whenever
$\mathcal{R}_e(c)$ requires attention with respect to $t$ at some stage $s+1$ we have
$K_{s+1}(W_s\restr_{2^{t+1}}\ |\ \alpha_e[s+1]\restr_{g(2^{t+1})})\leq c$ and there are at most $2^c$ many
descriptions of length $c$, between each $s_i$ and $s_{i+1}$ we must see an increase in $\alpha_e \restr_{g(2^{t+1})}$. 
If $R_e$ is not satisfied,
as we have assumed, then for each $t\in J_c(e)$ it will receive attention $|I_t|=2^t$ many times, meaning an increase in $\alpha_e \restr_{g(2^{t+1})}$ of at least $2^{t-c-g(2^{t+1})}$. 
Summing over all $t\in J_c(e)$ we have that 
\[
\alpha_e \geq \sum_{t\in J_c(e)} 2^{t-c-g(2^{t+1})}=
2^{-c}\cdot \sum_{t\in J_c(e)} 2^{t-g(2^{t+1})}
>2^{-c}\cdot 2^c=1.
\]
This is the required contradiction, since $\alpha\leq 1$. 
We conclude that each $\mathcal{R}_e(c)$ is met, which completes the proof of
clause (2).

\subsection{Proof of Theorem \ref{1OZQrzfDmp}}
Recall the following theorem.
\begin{repthm}{1OZQrzfDmp}
Let $g$ be a nondecreasing computable function such that 
$\sum_n 2^{-g(n)}$ converges to a 1-random real. Then
a \ce real is 1-random if and only if it computes the halting problem with 
respect to a Kolmogorov numbering with use $g$.
\end{repthm}

According to the hypothesis of Theorem \ref{1OZQrzfDmp}
and Lemma \ref{eq:condtest}
we have that
$\sum_i 2^{i-g(2^i)}$ is 1-random. Suppose that $A$ is the halting
problem with respect to a Kolmogorov numbering $(M_e)$ of Turing machines,
 and $\alpha$ is
a \ce real which computes $A$ with use
$g$. 
Let $U$ be the universal machine
whose halting problem is $A$, and let $f$ be a computable function such that
$n\mapsto f(e,n)$ is $\bigo{n}$ and such that
$U(f(e,n))\simeq M_e(n)$ for all $e,n$. 
By definition, we have $A=\{f(e,n)\ |\ M_e(n)\de\}$.
Moreover let $\Phi$ be a Turing functional with oracle use
uniformly bounded by $g(n)$ on each argument $n$, such that
$A=\Phi^{\alpha}$. Fix computable enumerations $(\Phi_s), (A_s)$ of $\Phi, A$ respectively. The proof proceeds much as in the first proof we gave for  clause (2) of Theorem \ref{FmKgYATdsh} previously. 
Once again we construct a machine $M$, and assume by the recursion theorem that we are given $b$ such that $M=M_b$. During the construction of $M$ we check the enumeration of $A$ as before, so as to ensure that there is no $n$ for which $f(b,n)$ is enumerated into $A$ but for which we have not yet defined $M$ on argument $n$, etc.  
The machine $M$ that we construct will be very simple, and it is the timing of enumerations into the domain of $M$ which is key.  Fix $c$ to be a constant such that
$f(b,n)\leq 2^{c-1}\cdot n$ for all $n$.
Again we define $I_t=[2^t-1,2^{t+1}-1)$ so that $|I_t|=2^t$.
This time, however, we say that $I_t$ requires attention at stage $s+1$, if $M(i)[s]$ is not defined for some $i\in I_t$,
and if also  $\Phi^{\alpha}(j)[s+1]=A(j)[s+1]$ for all $j\leq 2^{c-1}2^{t+1}=2^{c+t}$.
For the least $t$ which requires attention at stage $s+1$ (if any), we find the least $i\in I_t$ such that $M(i)[s]\uparrow $ and we define $M(i)\downarrow=0$.

The rough idea is now to define a sequence of stages $s_t$ such that if we define $\gamma_t=\alpha[s_t]$ and $\delta_t= 
\sum_{n=0}^{t-1} 2^{n+c-g(2^{n+c})}$, then 
\begin{equation} \label{want}
2^{c}\cdot(\gamma_{t+1}-\gamma_t)\geq (\delta_{t+1}-\delta_t).
\end{equation}
 According to the characterisation \eqref{pI7DoArmo} of Solovay reducibility, 
 this means that  the limit
$\delta$ of $(\delta_t)$ is Solovay reducible to the limit
 $\alpha$ of $(\gamma_t)$.
Since $\delta$ is 1-random, it then follows that $\alpha$ is 1-random, as required. We define $s_0=0$ and, for each $t>0$, we define $s_t$ to be the first stage at which $I_{t}$ receives attention. 
Each time $I_t$ receives attention, we must see an increase in $\alpha\restr_{g(2^{c-1}\cdot 2^{t+1})} = \alpha\restr_{g(2^{t+c})}$. 
Since $I_t$  will receive attention $|I_t|=2^t$ many times, this means a total increase in $\alpha\restr_{g(2^{t+c})}$ of at least $2^{t-g(2^{t+c})}=2^{-c}2^{t+c-g(2^{t+c})}$. We shall therefore have that (\ref{want}) holds, if we define 
$\gamma_t=\alpha\restr_{g(2^{t+c})}[s_t]$.

\section{Computing \ce reals from Omega numbers}
This section is devoted to the proof of Theorem \ref{yJAbgkLtG2}.
In Section \ref{hvsFRstJIc} we derive the first part of this result from
Theorem \ref{FmKgYATdsh}, while Section 
\ref{XLhNRMuqDV} contains a more sophisticated argument for the proof of the second part.

\subsection{Proof of Theorem \ref{yJAbgkLtG2}, clause (1)}\label{hvsFRstJIc}
Recall the following theorem.
\begin{repthm}{yJAbgkLtG2}[Clause (1)]
Let $h$ be computable.
If $\sum_n 2^{n-h(n)}$ converges, then for any $\Omega$-number $\Omega$ and any c.e.\ real $\alpha$, $\alpha$ can be computed from $\Omega$ with use bounded by $h$. 
\end{repthm}

Let $\Omega$ be an omega number, let $\alpha$ be 
a \ce real and let $g=h(n)-n$
be as in the statement of the first clause.
By the 
proof of clause (iv) of Theorem \ref{6MsEpxPF1}
in Section \ref{RwuumeDoKH}, there exists a \ce
set $A$ which  computes $\alpha$ with use $2^n$.
By Lemma \ref{eq:condtest} 
\begin{equation}\label{8uLwuGAAFK}
\sum_i 2^{-g(i)}<\infty
\Rightarrow
\sum_i 2^i\cdot 2^{-(\log 2^i + g(\log 2^i))}<\infty
\Rightarrow
\sum_i 2^{-(\log i + g(\log i))}<\infty
\end{equation}

so by Theorem \ref{FmKgYATdsh} the set $A$ is computable from $\Omega$ with
use bounded by $\log n + g(\log n)$.
By composing the two reductions we conclude that $\alpha$
is computable from $\Omega$ with use bounded by
$\log 2^n + g(\log 2^n)$ \ie
 $n+g(n)$.

We can use a more direct argument to prove the same fact, without the
hypothesis that $h(n)-n$ is non-decreasing.
Let $g,\Omega, \alpha$ be as above and let
$(\alpha_s), (\Omega_s)$ be computable nondecreasing 
dyadic rational approximations that converge to $\alpha,\Omega$ respectively.
We construct a Solovay test along with a c.e. set $I$,  as follows. 
At each stage $s+1$ we consider the least
$n$ such that $\alpha_s(n)\neq \alpha_{s+1}(n)$, if such exists. If there exists such an $n$ we define
$\sigma_{s}=\Omega_{s+1}\restr_{n+g(n)}$ and enumerate $s$ into $I$.
We verify that the set of strings $\sigma_s, s\in I$ is a Solovay test. Note that for every $n$,
the number of stages $s$ such that $n$ is the least number with the property that
 $\alpha_s(n)\neq \alpha_{s+1}(n)$ is bounded above by 
 the number of times that $\alpha_s(n)$ can change from
0 to 1 in this monotone approximation to $\alpha$. Hence this number is bounded above
by $2^{n}$. So:
\[
\sum_{s\in I} 2^{-|\sigma_s|}\leq 
\sum_n 2^n\cdot 2^{-g(n)-n}=\sum_n 2^{-g(n)}<\infty.
\]
Since $\Omega$ is \ml random, there exists some $s_0$ such that none of the strings $\sigma_s$
for $s\in I$ and $s>s_0$ are prefixes of $\Omega$. This means that whenever our construction enumerates
$s$ in $I$ because we find $n$ with $\alpha_s(n)\neq \alpha_{s+1}(n)$, 
there exists some later stage where the approximation to $\Omega\restr_{n+g(n)}$ changes.
Hence with oracle $\Omega\restr_{s+g(s)}$ we can uniformly compute $\alpha(n)$, so
$\alpha$ is computable from $\Omega$ with oracle use $h$.

\subsection{Proof of Theorem \ref{yJAbgkLtG2}, clause (2)}\label{XLhNRMuqDV}
Recall the following theorem.
\begin{repthm}{yJAbgkLtG2}[Clause (2)]
Let $h$ be computable and such that $h(n)-n$ is non-decreasing.
If $\sum_n 2^{n-h(n)}$ diverges  
then no $\Omega$-number can compute all c.e.\  reals
with use $h+\bigo{1}$. In fact, in this case there exist two c.e.\  reals such 
that no c.e.\  real can compute both of them with use $h+\bigo{1}$.
\end{repthm}

Given a computable function $h$ such that  $h(n)-n$ is non-decreasing and 
$\sum_i 2^{n-h(n)}$ diverges, we need to construct two c.e.\  reals such 
that no c.e.\ real can compute both of them with use $h+\bigo{1}$.
Our presentation will be neater, however,  if we consider the following elementary fact,
which allows one to ignore the constants in the previous statement.

\begin{lem}[Space lemma]\label{le:sdreTbL9} 
If $g$ is computable, non-decreasing and $\sum_n 2^{-g(n)}=\infty$
then there exists a computable non-decreasing function $f$ such that 
$\lim_n (f(n)-g(n))=\infty$ and 
$\sum_n 2^{-f(n)}=\infty$.
\end{lem}
\begin{proof}
Consider a computable increasing sequence $(n_i)$ such that $n_0=0$ and
\[
\sum_{n\in I_k} 2^{-g(n)}>2^k
\hspace{0.5cm}
\textrm{for all $k$}
\]
where $I_k=[n_k,n_{k+1})$.
Then for each $k$ and each $i\in [n_k, n_{k+1})$
define $f(i)=g(i)+k$. Then
\[
\sum_n 2^{-f(n)}\geq \sum_k \left(\sum_{n\in I_k}2^{-f(n)}\right)
=
\sum_k \left(2^{-k}\cdot \sum_{n\in I_k}2^{-g(n)}\right)
\geq 
\sum_k \left(2^{-k}\cdot 2^{k}\right)=\infty
\]
which concludes the proof.
\end{proof}

By Lemma \ref{le:sdreTbL9}, for the proof of the second clause of Theorem 
\ref{yJAbgkLtG2} it suffices to 
show that, given a computable non-decreasing function $h$ with  
$\sum_i 2^{n-h(n)}=\infty$, there exist two c.e.\  reals such 
that no c.e.\  real can compute both of them with use $h$.
We need to construct two \ce reals $\alpha,\beta$ such the requirement
\begin{equation}\label{43sbLID3NW}
\mathcal{R}(\Phi,\Psi,\gamma):\ \alpha\neq \Phi^{\gamma}
\hspace{0.3cm}\vee\hspace{0.3cm}
\beta\neq \Psi^{\gamma}
\end{equation}
is met for every triple $(\Phi,\Psi,\gamma)$ such that  $\Phi,\Psi$ are Turing functionals with
use $h$ and $\gamma$ is a c.e.\ real.
For the simple case where $h(n)=n+\bigo{1}$,
this was done by Yu and Ding in \cite{DingY04}, and a simplification of their argument
appeared in \cite{BL05.5}. 
The underlying method for this type of argument involves an amplification game, which we
discuss in Section \ref{jfLlVatCrP}. Then in Section \ref{QgTGsPtrZf} we build on these
ideas in order to produce a more sophisticated argument which deals with an arbitrary choice of
$h$ satisfying the hypothesis of the theorem.

\subsubsection{Amplification games}\label{jfLlVatCrP}
The reals $\gamma$ in requirements
\eqref{43sbLID3NW} will be given with a non-decreasing computable
rational approximation $(\gamma_s)$. Our task is to define suitable
computable approximations $(\alpha_s)$, $(\beta_s)$ to $\alpha,\beta$ respectively,
so that \eqref{43sbLID3NW} is met. The idea is that if (the approximation to)
$\alpha\restr_n$  changes at a stage where $\Phi^{\gamma}$
is defined to be an extension of (the previous version of) $\alpha\restr_n$, then either $\alpha\neq \Phi^{\gamma}$ or
(the approximation to) $\gamma\restr_{h(n)}$ needs to change at a later stage
(and similarly for $\beta$).
This gives us a way to drive  $\gamma$
to larger and larger values, if $\Phi,\Psi$ insist on mapping the 
current approximations of $\gamma$ to
the current approximations of $\alpha,\beta$ respectively.

In order to elaborate on this approach,
consider the following game between two players, which monotonically approximate two
\ce reals $\alpha$, $ \gamma$ respectively. 
Each of these numbers
has some initial value, and during the stages of the game values can only 
increase. If
$\alpha$ increases and $i$ is the leftmost position where an $\alpha$-digit 
change occurs, then $\gamma$ has to increase in such a way that some 
$\gamma$-digit at
a position $\leq h(i)$ changes. This game describes a Turing computation of $\alpha$ from
$\gamma$ with use $h$. If
$\gamma$ has to code two reals $\alpha,\beta$ then we get a similar game 
(where, say, at each stage only one of $\alpha,\beta$ can change). 
In order to block one of the reductions in \eqref{43sbLID3NW} the idea is to
devise a strategy which forces $\gamma$ to either stop
computing $\alpha,\beta$ in this way, or else grow to exceed the interval $(0,1)$
to which it is assumed to belong.

It turns out that in such games, there is a best strategy for $\gamma$, in the sense
that it causes the least possible increases in $\gamma$ while responding to the requests of the opponent(s).
We say that $\gamma$ follows the {\em least effort
strategy} if at each stage it increases by the least amount needed 
in order to satisfy the requirements of
the game.

\begin{lem}[Least effort strategy]\label{le:beststrat}
In a game where $\gamma$ has to follow instructions of the type `change a 
digit at position $\leq n$',  the least effort strategy is a best strategy for $\gamma$. 
In other words,  if a different strategy produces $\gamma'$
then at each stage $s$ of the game $\gamma_s\leq\gamma_s'$.
\end{lem}
\begin{proof}
We use induction on the stages $s$. We have that $\gamma_0\leq\gamma_0'$. If \( 
\gamma_s =\gamma_s' \) then it is clear from the definition of the least 
effort strategy that the induction hypothesis
will hold at stage \( s+1 \). So suppose otherwise. Then
$\gamma_s < \gamma_s'$  so that there will
be a position $n$ such that $0=\gamma_s(n)< \gamma_s'(n)=1$ and 
$\gamma_s\restr
n=\gamma_s'\restr n$. Suppose that $\gamma, \gamma'$
are forced to change at a  position $\leq t$  at stage $s+1$. If $t<n$ it is 
clear
that $\gamma_{s+1}\leq\gamma_{s+1}'$. Otherwise the leftmost change $\gamma$ 
can
be forced to
make is at position $n$. Once again $\gamma_{s+1}\leq\gamma_{s+1}'$.
\end{proof}

Lemma \ref{le:beststrat} allows us to assume
a fixed strategy for the opponent approximating $\gamma$, which reduces our analysis to the assessment of
a deterministic process, once we specify our strategy for the requests
that the approximations to $\alpha,\beta$ generate.
The following observation is also useful in our analysis.
\begin{lem}{\em (Accumulation)}\label{le:pass}
Suppose that in some game (e.g.\ like the above) $\gamma$ has to follow
instructions of the type `change a digit at position $\leq n$'. Although
$\gamma_0=0$, some $\gamma'$ plays the same game while starting with
$\gamma_0'=\sigma$ for a finite
binary expansion $\sigma$. If
$\gamma$ and $ \gamma'$ both use the least effort strategy 
and the sequence of instructions only ever demands
change at positions $>|\sigma|$ then 
$\gamma'_s=\gamma_s+\sigma$ at every stage $s$.
\end{lem}
\begin{proof}
By induction on $s$. For $s=0$ the result is obvious. Suppose that the 
induction
hypothesis holds at stage $s$. Then  $\gamma'_s,
\gamma_s$ have the same expansions
after position $|\sigma|$.
At $s+1$, some demand for a change at some position $>|\sigma|$ appears and
since
$\gamma, \gamma'$ look the same on these
positions, $\gamma'_s$ will need to increase by the same amount that
$\gamma_s$ needs to increase. So
$\gamma'_{s+1}=\gamma_{s+1}+\sigma$ as required.
\end{proof}

We are now ready to describe the strategy for the approximation of $\alpha,\beta$
restricted to an interval $(t-n,t]$. This strategy automatically induces a deterministic
response from $\gamma$ according to the least effort strategy.

\begin{defi}\label{EavAdjOPTD}
The $h$-load process in $(t-n,t]$ with $(\alpha,\beta,\gamma)$ is the process which starts with
$\alpha_0=\beta_0=\gamma_0=0$, and at each stage $2^{-t}$ is added alternately in
$\alpha$ and $\beta$. Moreover at each stage $s+1$, 
if $k$ is the least number such that $\alpha_{s+1}\restr_k\neq\alpha_{s}\restr_k$ or 
$\beta_{s+1}\restr_k\neq\beta_{s}\restr_k$, 
we add to $\gamma$ the least amount which can change $\gamma\restr_{h(k)}$.
\end{defi}

In more detail,  the $h$-load process in $(t-n,t]$ with $(\alpha,\beta,\gamma)$ is 
defined by the following instructions.
Assume that $\alpha, \beta, \gamma$ have initial value $0$.
Repeat the following instructions until $\alpha(i)=\beta(i)=1$ for all
$i\in(t-n,t]$.

	\begin{enumerate}[\hspace{0.5cm}(1)]
	\item If $s$ odd then, let $\alpha = \alpha + 2^{-t}$ and let $k$ equal the
	 leftmost position where a change occurs in $\alpha$. Also
	add to $\gamma$ the least amount which causes a change in $\gamma\restr_{h(k)}$.
	\item If $s$ even, let $\beta = \beta + 2^{-t}$ and let $k$ equal the
	 leftmost position where a change occurs in $\beta$.
	Also	add to $\gamma$ the least amount which causes a change in $\gamma\restr_{h(k)}$.
	\end{enumerate}

It is not hard to see that the above procedure describes how $\gamma$ 
evolves
when it tries to code $\alpha, \beta$ via Turing reductions with 
use $h$ and it uses the {\em least effort strategy}. Player $\gamma$ follows the 
{\em least effort strategy} when it increases by the least amount which can 
rectify the functionals holding its computations of $\alpha$ and $\beta$.
Note that in a realistic construction, the steps of the above strategy are supposed to happen
at stages where $\gamma$ currently computes the first $t+1$ bits of both $\alpha,\beta$ 
via $\Phi,\Psi$ respectively.

The following lemma says that in the special case where $h(n)=n+\bigo{1}$, the
$h$-load process is successful at forcing $\gamma$ to grow significantly.

\begin{lem}[Atomic attack]\label{prop:main}
Let $n>0$ and $h(x)=x+c$ for some constant $c$. 
For any $k>0$ the
$h$-load process in $(k,k+n]$ with $(\alpha,\beta,\gamma)$ 
 ends with $\gamma = n \cdot 2^{-k-c}$.
\end{lem}

\begin{proof}
By induction: for $n=1$ the result is obvious. Assume that the result holds 
for $n$. Now pick
$k>0$ and consider the attack using $(k-1,k+n]$. It is clear 
that up
to a stage $s_0$ this will be identical to the procedure with attack 
interval
$(k,k+n]$. By the induction hypothesis $\gamma_{s_0} = n 2^{-k-c}$
and $\alpha(i)=\beta(i)=1$ for all $i\in (k,k+n]$, while 
$\alpha(k)=\beta(k)=0$.
According to the next step $\alpha$ changes at position $k$ and this forces
$\gamma$ to increase by $2^{-k-c}$
since $\gamma$ has no $1$s to the right of position $k+c$. Then $\beta$ does 
the same  and since $\gamma$ still has no $1$s to the right of position $k+c$,
$\gamma$ has to increase by $ 2^{-k-c}$ once again. So far
\[
\gamma = n 2^{-k-c} + 2^{-k-c} + 2^{-k-c} = n 2^{-k-c} + 2^{-(k-1)-c}
\]
and $\alpha(i)=\beta(i)=0$ for all $i\in (k,k+n]$ while 
$\alpha(k)=\beta(k)=1$.
By applying the induction hypothesis again together with Lemma
\ref{le:pass}, the further increase of $\gamma$ will be exactly $n 2^{-k-c}$. 
So
\[
\gamma = n 2^{-k-c} + 2^{-(k-1)-c} + n 2^{-k-c} = (n+1)\cdot 2^{-(k-1)-c}
\]
as required.
\end{proof}
The analysis we just presented is sufficient for a construction of $\alpha,\beta$
which proves the second clause of
Theorem \ref{yJAbgkLtG2} for the specific case that $h(n)=n+\bigo{1}$.
One only has to assign attack intervals to different versions of requirement
\eqref{43sbLID3NW} and implement the $h$-load process in a priority fashion, 
gradually satisfying all conditions.
In the next section we build on these ideas in order to deal with the general case for $h$,
and in Section \ref{llohQmoEdz} we give the formal construction of $\alpha,\beta$.

\subsubsection{Generalized amplification games}\label{QgTGsPtrZf}
We wish to obtain a version of Lemma \ref{prop:main} which does not depend on a fixed choice of
$h$. For ease of notation, let $g(n)=h(n)-n$ and assume that $g$ is non-decreasing.

It is tempting to think that perhaps, in this general case, the $h$-load process
in $(k,k+n]$ with $(\alpha,\beta,\gamma)$  ends up with
\begin{equation}\label{BRHqKK3T8u}
\gamma=2^{-k}\cdot \sum_{i\in(k,k+n]} 2^{-g(i)}
\hspace{0.8cm}
\textrm{(false amplification lower bound).}
\end{equation}

With a few simple examples, the reader can be convinced that this is not the case.
Luckily, the $h$-load process does give a usable amplification effect, but in order to
obtain the factor in this amplification we need to calibrate the divergence of the series
$\sum_i 2^{-g(i)}$, by viewing $g$ as a step-function.

\begin{defi}[Signature of a non-decreasing function] \label{sig}
Let $g:\Nat\to\Nat$ be a non-decreasing function.
The {\em signature of $g$} is
a (finite or infinite) sequence $(c_j, I_j)$ of pairs of numbers $c_j$ and intervals $I_j$ (for $j\geq 0$),
such that $(I_j)$ is a partition of $\Nat$, 
$g(x)=c_j$ for
all $x\in I_j$ and all $j$, and $c_j> c_{j'}$ for $j> j'$.
The length of the sequence $(c_j, I_j)$ is called the length of the signature of $g$.
\end{defi}

Note that if $(c_j, I_j)$ is the signature of $g$ then
$\sum_i 2^{-g(i)}=\sum_j |I_j|\cdot 2^{-c_j}$, where in the latter sum the indices run over the length of
the sequence $(c_j, I_j)$. More generally, if $J$ is an interval of numbers which are less than the length
of the signature of $g$ we have
$\sum_{i\in I} 2^{-g(i)}=\sum_{j\in J} |I_j|\cdot 2^{-c_j}$ where $I=\cup_{j\in J} I_j$.
We define a truncation operation that applies to such partial sums, in order to replace the false
amplification lower bound \ref{BRHqKK3T8u} with a correct lower bound.

\begin{defi}[Truncation]
Given a real number $x\in (0,1)$ and an increasing sequence $(c_j)$, 
for each $t$ let
\begin{equation*}
T(x,c_t)=\sum_{i\leq t} n_i\cdot 2^{-c_i}
\hspace{0.5cm}
\textrm{where}
\hspace{0.5cm}
x= n_0\cdot 2^{-c_0}+ n_1\cdot 2^{-c_1}+\cdots=\sum_i n_i\cdot 2^{-c_i}
\end{equation*}
is the unique representation of $x$ as a sum of multiples of $2^{-c_j}$
such that $n_{i+1}\cdot 2^{-c_{i+1}}<2^{-c_i}$ for all $i$.
\end{defi}

 We now define a sequence of truncated sums that can be used in replacing the false 
 lower bound \eqref{BRHqKK3T8u}
 with a valid one. These quantities depend on the sequence $(c_j, I_j)$, which in turn depends on
 the function $g$.

\begin{defi}[Truncated sums]
Given the finite or infinite sequence $(c_j, I_j)$ and $k$ such that $c_j$ and $I_j$ are defined, the sequence of truncated sums with respect to $k$ is defined
inductively as follows ($i<k$):
\begin{eqnarray*}
S_k(0)&=&T(|I_k|\cdot 2^{-c_k}, c_{k-1})\\
S_k(i)&=&T(|I_{k-i}|\cdot 2^{-c_{k-i}}+S_{k}(i-1), c_{k-i-1}) \ \ \ \ \ \ \ \ \ \mbox{for }i>0. 
\end{eqnarray*}
\end{defi}

Before we show that the $h$-load strategy with $(\alpha,\beta,\gamma)$ forces $\gamma$ to
grow to a suitable truncated sum, we show that the truncated sums grow appropriately, assuming that
$\sum_i 2^{-g(i)}$ is unbounded.
Note that for increasing $i$, the sum $S_k(i)$ corresponds to later stages of the $h$-load
strategy, and intervals $I_{k-i}$ which are closer to the decimal point. This explains the
decreasing indices in the above definition and the following lemma.

\begin{lem}[Lower bounds on truncated sums]\label{UxBg9AKlSc}
For each $t<k$ we have
$S_k(t)\geq \sum_{i\leq t} |I_{k-i}|\cdot 2^{-c_{k-i}} -1$.
\end{lem}
\begin{proof}
Since $\sum_{i\in (c_{0},c_k]} 2^{-i}<1$ it suffices to show that for each $t<k$ we have
\begin{equation}\label{mQAbEePJda}
\sum_{i\leq t} |I_{k-i}|\cdot 2^{-c_{k-i}}\leq
S_k(t) +\sum_{i\in (c_{k-t-1},c_{k}]} 2^{-i}.
\end{equation}
We use finite induction on $t<k$. For $t=0$,
we have $|I_k|\cdot 2^{-c_k}-S_k(0)\leq \sum_{i\in (c_{k-1},c_k]} 2^{-i}$. 
Now inductively assume that \eqref{mQAbEePJda} holds for some $t$ such that $t<k-1$.
Then using the induction hypothesis we have:
\[
\sum_{i\leq t+1} |I_{k-i}|\cdot 2^{-c_{k-i}}=
\sum_{i\leq t} |I_{k-i}|\cdot 2^{-c_{k-i}} +
|I_{k-t-1}|\cdot 2^{-c_{k-t-1}}\leq
S_k(t) +\sum_{i\in (c_{k-t-1},c_{k}]} 2^{-i}+
|I_{k-t-1}|\cdot 2^{-c_{k-t-1}}.
\]
But
$S_{k}(t)+|I_{k-t-1}|\cdot 2^{-c_{k-t-1}}$ is clearly bounded above by
$T(S_{k}(t)+|I_{k-t-1}|\cdot 2^{-c_{k-t-1}}, c_{k-t-2})+\sum_{i\in (c_{k-t-2},c_{k-t-1}]} 2^{-i}$. So:
\[
\sum_{i\leq t+1} |I_{k-i}|\cdot 2^{-c_{k-i}}\leq
T(S_{k}(t)+|I_{k-t-1}|\cdot 2^{-c_{k-t-1}}, c_{k-t-2})+
\sum_{i\in (c_{k-t-2},c_{k}]} 2^{-i}=
S_k(t+1)+\sum_{i\in (c_{k-t-2},c_{k}]} 2^{-i},
\]
which concludes the induction step and the proof.
\end{proof}
Now, in general, the signature of $g$ may be finite or infinite. The case in which this signature is finite, however, corresponds to the situation $h(n)=n+\bigo{1}$, which was previously dealt with by  Yu and Ding in \cite{DingY04}, as discussed earlier. For ease of notation, \emph{we therefore assume from now on that the signature of $g$ is infinite}. 

Suppose that we apply the $h$-load process in
$I=\bigcup_{j\in [k-t,k]} I_j$. Then we want to show that $\gamma$ will  be larger than $2^{-m}\cdot S_k(t)$ by the end of the process, where $m$ is the least element of $I$. 
We will derive this 
(see Corollary \ref{ffj4W2aE1B})
from the following more general result, which can be proved by induction.
We note that in \eqref{jJI1ha5DBP}, in the trivial case where $t=0$ we let
$S_k(t-1):=0$.

\begin{lem}\label{XdtUDTsWV}
Let $g:\Nat\to\Nat$ be a non-decreasing computable function with signature $(c_j,I_j)$ (of infinite length), let
$h(x)=x+g(x)$ and suppose $t\leq k$. 
If $m$ is either inside $I_{k-t}$ or is the largest number which is less than all numbers
in $I_{k-t}$, then
the final value of $\gamma$ after the $h$-load process 
with $(\alpha,\beta,\gamma)$
in the interval
\begin{equation}\label{oE6vNBPbI}
(m,\max I_{k-t}]\cup (\bigcup_{i\in (k-t,k]} I_i)
\end{equation}
has the property that
\begin{equation}\label{jJI1ha5DBP}
T(2^{m}\cdot \gamma, c_{k-t})=S_k(t-1)+(\max I_{k-t}-m)\cdot 2^{-c_{k-t}}.
\end{equation}
\end{lem}
\begin{proof}
We use finite induction on the numbers $t\leq k$
and the numbers $m$ in $I_{k-t}\cup \{\min I_{k-t} -1\}$.
For the case when $t=0$, the $h$-load process occurs in the interval $(m,\max I_k]$,
where $m$ is either $\min I_k -1$ or any number in $I_k$.
In this case, by Lemma \ref{prop:main}
we have 
$T(2^{m}\cdot \gamma, c_{k})=(\max I_{k}-m)\cdot 2^{-c_{k}}$.
Therefore in this case \eqref{jJI1ha5DBP} holds.

Now consider some $t>0$ (and, of course, $t\leq k$) and any
$m\in [\min I_{k-t}-1, \max I_{k-t})$, and inductively assume that when the attack occurs in
the interval
\begin{equation}\label{xXZr9C1fqG}
(m+1,\max I_{k-t}]\cup (\bigcup_{i\in (k-t,k]} I_i)
\end{equation}
the final value of $\gamma$ satisfies
\begin{equation}\label{Cn4QkQm4AJ}
T(2^{m+1}\cdot \gamma, c_{k-t})=S_k(t-1)+(\max I_{k-t}-m-1)\cdot 2^{-c_{k-t}}.
\end{equation}

Note that this conclusion follows from the induction hypothesis, even in the case that $m=\max I_{k-t}-1$, since then the induction hypothesis gives:
\begin{equation}\label{Cn4QkQm4AO}
T(2^{m+1}\cdot \gamma, c_{k-(t-1)})=S_k(t-2)+|I_{k-(t-1)}| \cdot 2^{-c_{k-(t-1)}}.  
\end{equation}
Now if we let the r.h.s.\  equal $\delta$, then by definition $T(\delta,c_{k-t})=S_k(t-1)$, giving \ref{Cn4QkQm4AJ}, as required.

There are two qualitatively different cases regarding the value of $m+1$.
The first one is when $m+1=\max I_{k-t}$, which means that $m+1$ is the first
number in the latest interval  $I_{k-t}$. The other case is when
$m+1<\max I_{k-t}$ and $m+1\geq \min I_{k-t}$.
What is special in the first case is that half-way through the $h$-load process in 
the interval \eqref{oE6vNBPbI}, when $\alpha(m+1)$ changes, this is the first time
that some $\alpha(i)$ change requires $\gamma$ to change on digit $i+c_{k-t}$
(or before that). All previous requests on $\gamma$  required a change at  
$i+c_{k-t+s}$ for $s>0$. In other words, the case $m+1=\max I_{k-t}$ is special because 
it is when we cross into a new interval. Despite this apparent difference 
between the two cases, they can be dealt with uniformly, as we now show.

The attack on the interval \eqref{oE6vNBPbI} can be partitioned into three phases.
Let $\ell$ be the maximum of the interval \eqref{oE6vNBPbI}. 
Recall that each stage of this attack consists of adding $2^{-\ell}$ alternately to $\alpha$ 
and $\beta$.
The first phase of the attack consists of those stages up to the point where
$\alpha(m+1)=\beta(m+1)=0$ and
$\alpha(i)=\beta(i)=1$ for all $i$ in the interval
\eqref{xXZr9C1fqG}. According to the induction hypothesis,
at the end of the first phase of the attack \ref{Cn4QkQm4AJ} holds.
By the definition of $S_k(t-1)$, this number is an integer multiple of $2^{-c_{k-t}}$.
At the next stage of the attack, 
$\alpha(m+1)$ will turn from 0 to 1, forcing $\gamma$ to change at position $m+1+c_{k-t}$ or higher.
This means that $T(\gamma, m+1+ c_{k-t})$ will increase by $2^{-m-1-c_{k-t}}$, because it does not have
any 1s after position $m+1+c_{k-t}$.
In other words, 
$T(2^{m+1}\cdot\gamma, c_{k-t})$ will increase by $2^{-c_{k-t}}$.
Then 
$\beta(m+1)$ will turn from 0 to 1, 
forcing $\gamma$ to change again at position $m+1+c_{k-t}$ or higher.
This means that 
$T(2^{m+1}\cdot\gamma, c_{k-t})$ will increase by another $2^{-c_{k-t}}$.
The latter two stages make up the second phase of the attack. The third and final
phase consists of all remaining stages of the attack. 
At the end of the second phase $T(2^{m+1}\cdot\gamma, c_{k-t})$ has increased by
\[
2\cdot 2^{-c_{k-t}} +S_k(t-1)+(\max I_{k-t}-m-1)\cdot 2^{-c_{k-t}}
\]
compared to its initial value (which was 0).
By applying the
induction hypothesis and the accumulation lemma
(Lemma \ref{le:pass}), at the end of the attack
 $T(2^{m+1}\cdot\gamma, c_{k-t})$  then equals:
\[
2\cdot 2^{-c_{k-t}} +2\cdot (S_k(t-1)+(\max I_{k-t}-m-1)\cdot 2^{-c_{k-t}})
=
2\cdot (S_k(t-1)+(\max I_{k-t}-m)\cdot 2^{-c_{k-t}}).
\]
By definition
$S_k(t-1)$ is an integer multiple of $2^{-c_{k-t}}$.  The equation above therefore establishes 
\eqref{jJI1ha5DBP} as required, since  if $T(2^{m+1}\cdot\alpha,c)=2\cdot C$ 
where $C$ is an integer multiple of $2^{-c}$, then
$T(2^{m}\cdot\alpha,c)=C$.
This concludes the induction step and the proof of the lemma.
\end{proof}

The main construction employs $h$-load processes in unions of consecutive intervals $I_j$
from the signature of $g$. So we only need the following special case of Lemma \ref{XdtUDTsWV}.

\begin{coro}[Lower bound]\label{ffj4W2aE1B}
Let $g:\Nat\to\Nat$ be a non-decreasing computable function with infinite signature $(c_j,I_j)$, and let
$h(x)=x+g(x)$.
The final value of $\gamma$ after the $h$-load process 
with $(\alpha,\beta,\gamma)$
in the interval $I=\bigcup_{j\in [k-t,k]} I_j$  is at least
$2^{-m}\cdot S_k(t)$, where $m$ is the largest number less than all elements of $I$.
\end{coro}
\begin{proof}
Since
$m$  is the largest number which is less than all numbers
in $I_{k-t}$, we have $|I_{k-t}|=\max I_{k-t}-m$, so by the definition of the reduced sums
and Lemma \ref{XdtUDTsWV} we get
\begin{equation*}
T(2^{m}\cdot\gamma, c_{k-t})=S_k(t-1)+|I_{k-t}|\cdot 2^{-c_{k-t}}
\end{equation*}
so
\begin{equation*}
T(2^{m}\cdot\gamma, c_{k-t-1})\geq
T(T(2^{m}\cdot\gamma, c_{k-t}),c_{k-t-1})=
T(S_k(t-1)+|I_{k-t}|\cdot 2^{-c_{k-t}},c_{k-t-1})=S_k(t),
\end{equation*}
which means that $\gamma\geq 2^{-m}\cdot S_k(t)$.
\end{proof}

We are ready to describe an explicit construction of two reals
$\alpha,\beta$ establishing the second clause of Theorem \ref{yJAbgkLtG2} for the (remaining) case that $g$ has infinite signature.

\subsubsection{Construction of the two reals and verification}\label{llohQmoEdz}
Let $g$ be a non-decreasing computable function with infinite signature such that 
$\sum_i 2^{-g(i)}=\infty$, and let $h(x)=x+g(x)$.
Let $(\Phi_e,\Psi_e,\gamma_e)$ be an effective enumeration of all triples of
Turing functionals $\Phi_e,\Psi_e$, and \ce reals $\gamma_e$.
According to the discussion at the beginning of Section \ref{XLhNRMuqDV}, for the proof of
 the second clause of Theorem \ref{yJAbgkLtG2},  it suffices to construct two \ce reals $\alpha,\beta$
 such that the following requirements are met.
\begin{equation}\label{e1mYBhRnKW}
\mathcal{R}_e:\ \alpha\neq \Phi_e^{\gamma_e}
\hspace{0.3cm}\vee\hspace{0.3cm}
\beta\neq \Psi_e^{\gamma_e}.
\end{equation}
Let $(c_j, I_j)$ be as specified in Definition \ref{sig}. The construction consists of assigning appropriate intervals $J_e$ to the requirements
$\mathcal{R}_e$ and implementing the $h$-load process with $(\alpha,\beta,\gamma_e)$
in $J_e$ independently for each $e$, in a priority fashion.
We start with the definition of the intervals $J_e$, which is informed by the lower bound
established in Corollary \ref{ffj4W2aE1B}.
We define a computable increasing sequence $(n_j)$ inductively and define:
\[
J_e=\bigcup_{j\in (n_e, n_{e+1}]} I_j. 
\]
Let $n_0=1$. Given $n_e$, let
$n_{e+1}$ be the least number which is greater than $n_e$ and such that if $m$ is the largest number in $I_{n_e}$ then: 
\[
2^{-m}\cdot S_{n_{e+1}}(n_{e+1}-n_e)>1.
\]
According to our hypothesis concerning $g$ and Lemma \ref{UxBg9AKlSc}, such a number
$n_{e+1}$ exists, so the definition of $(n_j)$ is sound.
Fix a universal enumeration with respect to which we can
approximate the Turing functionals 
$\Phi_e,\Psi_e$ and the \ce reals $\gamma_e$.
The suffix `$[s]$' on a formal expression means the approximation of the expression at
stage $s$ of the universal enumeration. The stages of the construction are the stages of the
universal enumeration.
Recall from Definition \ref{EavAdjOPTD} that each step
in the $h$-load process on $J_e$ consists of adding $2^{-n_{e+1}}$ alternately
to $\alpha$ and $\beta$.
According to this process, we will only allow
$\mathcal{R}_e$ to act at most $2^{n_{e+1}-n_e}-1$ times for each of $\alpha, \beta$.
We say that strategy $\mathcal{R}_e$ is {\em active} at stage $s$ if:
\begin{enumerate}[\hspace{0.5cm}(a)]
\item $\alpha(t)[s]= \Phi_e^{\gamma_e}(t)[s]$ and 
$\beta(t)[s]= \Psi_e^{\gamma_e}(t)[s]$
for all $t\in\cup_{i\leq e} J_i$.
\item strategy $\mathcal{R}_e$ has acted less than $2\cdot (2^{n_{e+1}-n_e}-1)$ 
many times in previous stages. 
\end{enumerate}

\ \paragraph{{\bf Construction}.}
At each stage $s$, if there exists no $e$ such that $\mathcal{R}_e$ is active, then go to the next stage. Otherwise, let $e$ be the least such. We say that
$s$ is an `$e$-stage'. If there has not been a previous $e$-stage, or if $\beta$ was increased at the most recent $e$-stage, then 
add $2^{-n_{e+1}}$ to $\alpha$ and say that $\mathcal{R}_e$ acts on $\alpha$ at stage $s$. 
Otherwise add $2^{-n_{e+1}}$ to $\beta$ and say that $\mathcal{R}_e$ acts on $\beta$. 
Go to the next stage.

We now verify that the construction produces approximations to reals $\alpha,\beta$ such that
requirement $\mathcal{R}_e$ is met for each $e$. Note that
$\mathcal{R}_e$ 
is only allowed to act at most $2^{n_{e+1}-n_e}-1$ times for each of $\alpha, \beta$.
This means that no action of $\mathcal{R}_e$ can change a digit of $\alpha$ or $\beta$ which is
outside $J_e$. In particular, the approximations to $\alpha$ and $\beta$ that are defined in the construction
converge to two reals in $[0,1]$.

Finally, we show that  for each $e$, requirement $\mathcal{R}_e$ is met and is active 
at only finitely many stages.
Inductively suppose that this is the case for $\mathcal{R}_i$, $i<e$. 
Note that $\mathcal{R}_e$ can only be active when it has
acted less than $2\cdot (2^{n_{e+1}-n_e}-1)$ many times.
Therefore, by the construction and the induction hypothesis,
$\mathcal{R}_e$ is only active at finitely many stages.
It remains to show that $\mathcal{R}_e$ is met. Towards a contradiction, suppose that this is not the case.
Then for some \ce real $\gamma_e'$ (the one that follows the least effort 
strategy with respect to $\alpha,\beta$) throughout a subsequence $s_i, i<k$ 
of the stages of the 
construction,
the triple $(\alpha,\beta,\gamma_e')$ follows a complete $h$-load process
in $J_e$. By 
Corollary \ref{ffj4W2aE1B}, 
it follows that $\gamma_e'\geq 2^{-n_e}\cdot S_{n_{e+1}}(n_{e+1}-n_e)$
and by the choice of $J_e$ we have $\gamma_e'>1$. 
Since the functionals $\Phi_e,\Psi_e$ have use-function $x+g(x)$,
if $\gamma_e \neq \gamma_e'$ then the approximations to $\gamma_e$ in the stages $s_i, i<k$ can be seen as a suboptimal
response to the changes of the bits of $\alpha,\beta$ in $J_e$
according to the amplification games of Section \ref{jfLlVatCrP}.
In particular, according to Lemma \ref{le:beststrat} and Lemma \ref{le:pass},
we have $\gamma_e[s_i]\geq \gamma_e'[s_i]$ for each $i<k$. This means that
$\gamma_e>1$, which contradicts the fact that $\gamma_e\in [0,1]$.
We conclude that  $\mathcal{R}_e$ is satisfied. 

\section{Proof of Theorem \ref{gGr7djFsLI}}
Recall the following theorem.
\begin{repthm}{gGr7djFsLI}
Let $\mathbf{a}$ be a \ce Turing degree and let $\Omega$ be (an instance of)
Chaitin's halting probability.
\begin{enumerate}[(1)]
\item If $\mathbf{a}$ is array computable, then every real in $\mathbf{a}$
is computable from $\Omega$ with identity use.
\item Otherwise there exists a \ce real in $\mathbf{a}$ which is not computable from $\Omega$
with use $n+\log n$.
\end{enumerate}
\end{repthm}

Theorem \ref{gGr7djFsLI} can be obtained directly by an application of the methods in
\cite[Sections 4.2 and 4.3]{IOPORT.05678491} to the
construction in the proof of Theorem \ref{yJAbgkLtG2} (b).
Since there are no new ideas involved in this application, 
and in the interest of space, we merely sketch the argument.
In \cite[Section 4.2]{IOPORT.05678491} 
it was shown that if $\alpha$ is in a \ce array computable degree, then it is computable
from $\sum_n g(n)\cdot 2^{-n}$ with identity use, where $g=\bigo{n}$ is a function with a nondecreasing
computable approximation. It is not hard to see that this implies that $\alpha$ is computable from
$\Omega$ with identity use, which is clause (1) of 
Theorem \ref{gGr7djFsLI}.

In \cite[Section 4.3]{IOPORT.05678491} it was shown that if $\mathbf{a}$ is an array noncomputable
\ce degree, then there exist two \ce reals in such that no \ce real can compute both of them with
use $n+\bigo{1}$. This result was obtained by the application of array noncomputable permitting to
the construction in \cite{DingY04} (which obtained two \ce reals with the above property).
The same can be done in entirely similar way to the construction
in the proof of Theorem \ref{yJAbgkLtG2} (2), which is structurally similar to the construction in
 \cite{DingY04}. This extension of our argument in Section 
\ref{XLhNRMuqDV} shows that the two \ce reals of
the second clause of Theorem \ref{yJAbgkLtG2} (1) can be found in any array noncomputable
\ce degree. If we choose $h(n)=n+\log n$ in this theorem, we obtain 
the second clause of Theorem \ref{gGr7djFsLI}.

\section{Conclusion and a conjecture}
Chaitin's omega numbers are known to compute the solutions to many interesting problems,
and to do so with  access only to a short initial segments of the oracle.
Although $\Omega$ contains the same information as the halting problem, this information is much more
tightly packed into short initial segments. 
Despite these well-known facts, little was known about the number of bits of $\Omega$ that
are needed to compute  halting probabilities or  \ce sets, and in particular the 
asymptotically optimal use of the oracle in such computations.

In this work we provide answers to these questions, and expose various connections with
current research in algorithmic randomness.
Moreover, our results point to several other open problems
which appear to be an interesting direction of research.

Barmpalias and Lewis-Pye \cite{jflBaiasL06} showed that there exists a \ce real which is not
computable from any omega number with use $n+\bigo{1}$.
Another proof of this result (which was used to obtain a characterisation of the degrees
of the \ce reals with this property) was given in
\cite[Section 4.3]{IOPORT.05678491}.
We conjecture that this holds more generally in the spirit of this paper.
In particular, if $g$ is a computable non-decreasing function such that
$\sum_i 2^{-g(i)}=\infty$, we conjecture that there exists a \ce real which is not computable
by any omega number with use $g(n)+n$.


\end{document}